\documentclass[11pt]{article}
\usepackage{amsmath,amssymb,amsthm,amsfonts,amstext,amsbsy,amscd}
\usepackage{graphics}
\usepackage{graphicx}
\usepackage{float}

\usepackage{dsfont}
\usepackage{color}

\RequirePackage[colorlinks,citecolor=blue,urlcolor=blue]{hyperref}
\RequirePackage{hypernat}

\newcommand{\PP}{\mathbb{P}}
\newcommand{\Z}{\mathbb{Z}}
\newcommand{\E}{\mathbb{E}}

\newcommand{\R}{\mathbb{R}}
\newcommand{\N}{\mathbb{N}}

\newcommand{\V}{\mathbb{V}}

\newtheorem{theorem}{Theorem}
\newtheorem{remark}{Remark}
\newtheorem{lemma}{Lemma}
\newtheorem{definition}{Definition}

\newtheorem{proposition}{Proposition}
\newtheorem{assumption}{Assumption}

\begin{document}

\title{Nonparametric estimation of a renewal reward process from discrete data}
\author{C\'eline Duval\footnote{GIS-CREST and CNRS-UMR 8050, 3, avenue Pierre Larousse, 92245 Malakoff Cedex, France.}}

\date{}
\maketitle

\begin{abstract}
We study the nonparametric estimation of the jump density of a renewal reward process from one discretely observed sample path over $[0,T]$. We consider the regime when the sampling rate $\Delta=\Delta_T\rightarrow0$ as $T\rightarrow\infty$. The main difficulty is that a renewal reward process is not a L\'evy process: the increments are non stationary and dependent. We propose an adaptive wavelet threshold density estimator and study its performance for the $L_p$ loss, $p\geq 1$, over Besov spaces. We achieve minimax rates of convergence for sampling rates $\Delta_T$ that vanish with $T$ at polynomial rate. In the same spirit as Buchmann and Gr\"ubel (2003) and Duval (2012), the estimation procedure is based on the inversion of the compounding operator. The inverse has no closed form expression and is approached with a fixed point technique.
\end{abstract}
\noindent \textbf{AMS 2000 subject classifications:} 62G99, 62M99, 60G50.\\
\noindent \textbf{ Keywords:} Renewal reward process, Continuous time random walk, Compound Poisson process, Discretely observed random process, Wavelet density estimation.

\section{Introduction}
\subsection{Motivation and statistical setting}

Renewal reward processes are pure jump processes used in many application fields, for instance in seismology (see Alvarez \cite{Alvarez} or Helmstetter \textit{et al.} \cite{Helmstetter}), to model rainfall (see Rodriguez-Iturbe \textit{et al.} \cite{Rod}) or in mathematical insurance and finance (see for instance Scalas \textit{et al.} \cite{scalas05,Scalas} or Masolivier \textit{et al.} \cite{Masoliver}). If many papers are devoted to the estimation of a discretely observed L\'evy process (see for instance Bec and Lacour \cite{Lacour}, Comte and Genon-Catalot \cite{Comte09,Comte11}, Figueroa-L\'opez \cite{Lopez} and Duval \cite{Du2} for the high frequency case and Neumann and Rei\ss \cite{Reiss} and Comte and Genon-Catalot \cite{Comte10} for the low frequency one), to the knowledge of the author, little exists on the estimation of a discretely observed renewal reward process. Vardi \cite{Vardi} estimates the density of a renewal process without rewards from the continuous observation of several independent trajectories. In this paper we estimate the compound law of a renewal reward process when one trajectory is observed at a sampling rate that goes to 0 arbitrarily slowly.

Let $J_1,\dots,J_i$ be nonnegative independent random variables where $J_2,\dots,J_i$ are identically distributed. Define $T_i$ the time of the $i$th jump as $T_i=J_1+...+J_i,$ $i\geq1.$ The associated counting process or renewal process $R$ is
$$R_t=\sum_{i=1}^\infty\mathds{1}_{T_i\leq t},\ \ \ t\geq 0.$$ The Poisson process is a particular case of a renewal process, corresponding to exponentially distributed interarrivals $\big(J_i\big)$. That latter case excepted, $R$ does not have independent increments and is usually not stationary \textit{i.e.} for all positive $t,h$ the law of $R_{t+h}-R_t$ depends on $t$. Assume that the common distribution $\tau$ of the $\big(J_i\big)$ has finite expectation $$\mu=\int_0^\infty t\tau(dt)<\infty,$$ define the distribution \begin{align}\label{eq tau0}
\tau_0(x)&=\frac{1-\int_0^x\tau(dt)}{\mu}.
\end{align}The process $R$ is stationary if and only if $J_1$ has distribution $\tau_0$ (see Lindvall \cite{Lindvall} p.70). Define the renewal reward process $X$ as \begin{align*}
X_t&=\sum_{i=1}^{R_t}\xi_i,\ \ \ t\geq 0
\end{align*}
where the $\big(\xi_i\big)$ are independent and identically distributed random variables, independent of the interarrivals $\big(J_i\big)$. Renewal reward processes also correspond to decoupled continuous time random walks.

 Assume that we have discrete observations of the process $X$ over $[0,T]$ at times $i\Delta$ for some $\Delta>0$
\begin{align}\label{eq data}
 \big(X_\Delta,\ldots,X_{\lfloor T\Delta^{-1}\rfloor\Delta}\big).
\end{align} We focus on the \textit{microscopic regime} namely $$\Delta=\Delta_T \rightarrow 0 \ \ \ \ \ \mbox{ as }\ T\rightarrow\infty,$$ and work under the following assumption. \begin{assumption}\label{ass f}
The law of the $\xi_i$ has density $f$ which is absolutely continuous with respect to the Lebesgue measure. \\
\noindent The law of the $J_i$, $i\geq 2$ has density $\tau$ which is absolutely continuous with respect to the Lebesgue measure and $J_1$ has density $\tau_0$.
\end{assumption}
\noindent The necessity of the last part of Assumption \ref{ass f} is discussed in Section \ref{section discuss}.

We denote by $\mathcal{F}(\R)$ the space of densities with respect to the Lebesgue measure supported by $\R$. We investigate the nonparametric estimation of the density $f$ on a compact interval $\mathcal{D}$ of $\R$ from the observations \eqref{eq data}. To that end we use wavelet threshold density estimators and study their rate of convergence, uniformly over Besov balls, for the following loss function \begin{align}\label{eq loss}  \big(\E\big[\|\widehat {f}-f\|_{L_p(\mathcal{D})}^p\big]\big)^{1/p},\end{align} where $\widehat{f}$ is an estimator of $f$, $p\geq1$ and $\|.\|_{L_p(\mathcal{D})}$ denotes $L_p$ loss over the compact set $\mathcal{D}$. We do not assume the interarrival distribution $\tau$ to be known: it is a nuisance parameter.

We estimate $f$ from the increments of $X$, which are dependent. By Assumption \ref{ass f}, on the event $\{X_{i\Delta}-X_{(i-1)\Delta}=0\}$ no jump occurred between $(i-1)\Delta$ and $i\Delta$ so that the increment $X_{i\Delta}-X_{(i-1)\Delta}$ gives no information on $f$. In the microscopic regime $\Delta=\Delta_T\rightarrow 0$ many increments are zero, therefore to estimate $f$ we focus on the nonzero increments. We denote by $N_T$ their number over $[0,T]$. In that statistical context different difficulties arise; the number of data $N_T$ used for the estimation is random, the increments are dependent, but more importantly on the event $\{X_{i\Delta}-X_{(i-1)\Delta}\ne0\}$, the density of $X_{i\Delta}-X_{(i-1)\Delta}$ is not $f$. Indeed even if $\Delta$ is small there is always a positive probability that more than one jump occurred between $(i-1)\Delta$ and $i\Delta$. Conditional on $\{X_{i\Delta}-X_{(i-1)\Delta}\ne0\}$, the law of $X_{i\Delta}-X_{(i-1)\Delta}$ has density given by (see Proposition \ref{PropDefOperator} below) \begin{align}\label{eq operator 1}\mathbf{P}_\Delta[f](x)=\sum_{m=1}^\infty\PP\big(R_\Delta=m\big|R_\Delta\ne0\big)f^{\star m}(x), \ \ \ \ \mbox{ for }x\in\R,\end{align} where $\star $ is the convolution product and $f^{\star m}=f\star\ldots\star f$, $m$ times. Hereafter Lemma \ref{lem pmControl} gives for $\Delta$ small enough \begin{align}\label{eq lem1}1-2\tau(0)\Delta\leq \PP\big(R_\Delta=m\big|R_\Delta\ne0\big)\leq 1.\end{align}
We deduce from \eqref{eq lem1} the decomposition $$\mathbf{P}_\Delta[f]=f+r(\Delta),$$ where $r(\Delta)$ is a deterministic remainder of the order of $\Delta$. We will see in Theorem \ref{thm Renewal 1} that if $\Delta=\Delta_T$ goes to 0 fast enough, namely $T\Delta_T^2=O(1)$ (up to logarithmic factor in $T$) $r(\Delta)$ is negligible and it is possible to estimate $f$ with optimal rates by ignoring the remainder $r(\Delta)$. Otherwise, when there exists $0<\delta< 1$ such that $T\Delta_T^2=O(T^\delta)$ (up to logarithmic factors in $T$) the remainder $r(\Delta)$ is no longer negligible. The condition $\delta<1$ ensures that $\Delta_T$ goes to 0 as $T$ tends to infinity. In the sequel we distinguish two different regimes that will be treated separately.\begin{itemize}
\item \textit{Fast microscopic rates} when --up to logarithmic factors in $T$-- $$T\Delta_T^2=O(1).$$
\item \textit{Slow microscopic rates} when there exists $0<\delta<1$ such that --up to logarithmic factors in $T$-- $$T\Delta_T^2=O(T^\delta).$$
\end{itemize} Since all the results of the paper are given up to logarithmic factors in $T$, fast and slow microscopic rates cover all vanishing behaviours for $\Delta=\Delta_T$. We try to answer the following question: Is it possible to construct an adaptive wavelet estimator of $f$ in fast and slow microscopic regimes which is optimal? Papers which estimate nonparametrically the L\'evy measure from a discretely observed L\'evy process attain optimal rate estimators only for fast microscopic rates (see for instance Bec and Lacour \cite{Lacour}, Comte and Genon-Catalot \cite{Comte09,Comte10,Comte11} and Figueroa-L\'opez \cite{Lopez}).

\subsection{Our Results\label{section our res}}

In Section \ref{section fast} we estimate $f$ in the fast microscopic regime, the estimation procedure is based on the approximation $$f\approx\mathbf{P}_\Delta[f].$$ We construct an adaptive wavelet threshold density estimator from the observations \eqref{eq data}. It achieves the minimax rate of convergence which is $T^{-\alpha(s,p,\pi)}$ if $f$ is of regularity $s$ measured with the $L_\pi$ norm, $\pi>0$, and where $\alpha(s,\pi,p)\leq 1/2$ (see \eqref{eq alpha} hereafter). That procedure does not depend on the interarrival density $\tau$ apart from Assumption \ref{ass f}. Moreover the estimator does not explicitly depend on the random quantity $N_T$, the number of nonzero increment.

In Section \ref{section slow} we estimate $f$ in the slow microscopic regime, the estimation procedure is the analogue of the one used in Duval \cite{Du2}. The starting point is that $$f=\mathbf{P}_\Delta^{-1}\big[\mathbf{P}_\Delta[f]\big],$$ and we proceed in two steps to estimate $f$. The first step is the computation of the inverse of the operator $\mathbf{P}_\Delta$ defined in \eqref{eq operator 1}. That step can be referred as decoumpounding as introduced in Buchmann and Gr\"ubel \cite{Buchmann} or van Es \textit{et al.} \cite{van es}. That inverse cannot be explicitly calculated, contrary to \cite{Du2}, but can be approached using a fixed point method. Indeed $f$ is a fixed point of the operator
\begin{align*}
\mathbf{H}_{\Delta,f}:h\rightarrow \mathbf{P}_\Delta[f]+h-\mathbf{P}_\Delta[h]
\end{align*} which is a contraction if $h$ and $f$ verifies suitable smoothness properties (see Proposition \ref{prop complet contract} below). The Banach fixed point theorem guarantees that for $K$ in $\N$ and $p\geq 1$, $$\big\|\mathbf{H}_{\Delta,f}^{\circ K}\big[\mathbf{P}_\Delta[f]\big]-f\big\|$$ is small in a sense that we precise later. Next we observe that the Taylor expansion of order $K$ in $\Delta$ of $\mathbf{H}_{\Delta,f}^{\circ K}\big[\mathbf{P}_\Delta[f]\big]$ takes the form \begin{align}\label{eq approx intro}\sum_{m=1}^{K+1} l_m(\Delta)\mathbf{P}_\Delta[f]^{\star m},\end{align} where the $\big(l_m(\Delta)\big)$ depend on the unknown interarrival density $\tau$ (see Proposition \ref{PropDefOperator} below). If $\tau$ is described by an unknown parameter $\vartheta\in \R$ then $l_m(\Delta)=l_m(\Delta,\vartheta)$ is estimated by plugging an estimator of $\vartheta$.

The second step consists in estimating the densities $\mathbf{P}_\Delta[f]^{\star m}$, for $m=1,\ldots,K+1$. For that we focus on the $N_T$ nonzero increments which have density $\mathbf{P}_\Delta[f]$. The difficulty here is that we have $N_T$ dependent observations where $N_T$ is a random sum of dependent variables. The dependency of the increments is treated using that at each renewal times the renewal process forgets its past. To cope with the randomness of $N_T$, we prove that $N_T/T$ concentrates for $T$ large enough around a deterministic limit using Bernstein type inequalities for dependent data (see Lemma \ref{lem Nconcentre Renewal} in Section \ref{Section proof} and Dedecker \textit{et. al.} \cite{Doukhan}). In Theorem \ref{thm Renewal 2} we show that wavelet threshold estimators of $\mathbf{P}_\Delta[f]^{\star m}$ attain a rate of convergence --up to logarithmic factors-- in $T^{-\alpha(s,\pi,p)}$. We inject those estimators into \eqref{eq approx intro} and obtain an estimator of $f$ that we call \textit{estimator corrected at order $K$.}

The study of the rate of convergence of the estimator corrected at order $K$ requires to control two distinct error terms. A deterministic one due the first step which is the error made when approximating $f$ by \eqref{eq approx intro}. And a statistical one due to the replacement of the ${\mathbf{P}_\Delta}[f]^{\star m}$ by estimators in the second step. The deterministic error decreases when $K$ increases. We choose $K$ sufficiently large for the deterministic error term to be negligible in front of the statistical one. We give in Theorem \ref{thm Renewal 2} an upper bound for the rate of convergence of the estimator corrected at order $K$ which is in --up to logarithmic factors--
$$\max\{T^{-\alpha(s,\pi,p)},\Delta_T^{K+1}\}.$$ Since $\alpha(s,\pi,p)\leq 1/2$ if there exists $K_0$ such that \begin{align*}
T\Delta_T^{2K_0+2}\leq 1,\end{align*} the estimator corrected at order $K_0$ attains the optimal rate.

\begin{remark}There is a slight difference of methodology between fast and slow microscopic rates to estimate $f$; for fast rates we estimate $f$ using all the increments but in slow rates we focus on nonzero ones. In that latter case, building an estimator using all the increments, even zero ones, achieving the rates of Theorem \ref{thm Renewal 2} is possible but numerically unstable. And a technical constraint in the proof of the concentration of $N_T/T$ prevented us from having a unified procedure for fast and slow microscopic rates.  \end{remark}

The paper is organised as follows. In Section \ref{section fast} we give an adaptive minimax estimator of $f$ in the fast microscopic regime. In that Section we also define wavelet functions and Besov spaces that are used for the estimation and describe the law of the increments. Those results are also used in Section \ref{section slow} where we give an adaptive minimax estimator of $f$ in the slow microscopic regime. In both Sections \ref{section fast} and \ref{section slow} we give upper bounds for the rate of convergence of the estimator of $f$ for the $L_p$ loss defined in \eqref{eq loss}, $p\geq1$, uniformly over Besov balls. In Section \ref{section num ex}, a numerical example illustrates the behavior of the estimators of $f$ introduced in Sections \ref{section fast} and \ref{section slow}. Finally Section \ref{Section proof} is dedicated to the proofs.

\section{Estimation of $f$ in the fast microscopic regime\label{section fast}}

\subsection{Preliminary on Besov spaces and wavelet thresholding}
For the estimation, we use wavelet threshold density estimators and study their performance uniformly over Besov balls. In this paragraph we reproduce some classical results on Besov spaces, wavelet bases and wavelet threshold estimators (see Cohen \cite{Cohen}, Donoho \textit{et al.} \cite{Donoho96} or Kerkyacharian and Picard \cite{KP00}) that we use in the next sections.

\subsubsection*{Wavelets and Besov spaces}
We describe the smoothness of a function with Besov spaces on $\mathcal{D}$. We recall here some well documented results on Besov spaces and their connection to wavelet bases (see Cohen \cite{Cohen}, Donoho \textit{et al.} \cite{Donoho96} or Kerkyacharian and Picard \cite{KP00}). Let $\big(\psi_{\lambda}\big)_\lambda$ be a regular wavelet basis adapted to the domain $\mathcal{D}$. The multi-index $\lambda$ concatenates the spatial index and the resolution level $j=|\lambda|$. Set $\Lambda_j:=\{\lambda,|\lambda|=j\}$ and $\Lambda=\cup_{j\geq -1}\Lambda_j$, for $f$ in $L_p(\R)$ we have
\begin{align}\label{eq fdecomp1}
f&=\sum_{j\geq-1}\sum_{\lambda\in\Lambda_j}\langle f,\psi_\lambda\rangle\psi_\lambda,
\end{align} where $j=-1$ incorporates the low frequency part of the decomposition and $\langle .,\rangle$ denotes the usual $L_2$ inner product. For $s>0$ and $\pi \in (0,\infty]$ a function $f$ belongs to the Besov space $\mathcal{B}^s_{{\pi} \infty}(\mathcal{D})$ if the norm
\begin{align}\label{eq besov norm}\|f\|_{\mathcal{B}^s_{{\pi} \infty}(\mathcal{D})}:= \|f\|_{L_\pi(\mathcal{D})}+\|f^{(n)}\|_{L_\pi(\mathcal{D})}+\Big\|\frac{w_{\pi}^2(f^{(n)},t)}{t^a}\Big\|_{L_\infty(\mathcal{D})}
\end{align} is finite, where $s=n+a$, $n\in\N$ and $a\in (0,1]$, $w$ is the modulus of continuity defined by $$w_{\pi}^2(f,t)=\underset{|f|\leq t}{\sup}\big\|\mathbf{D} ^h\mathbf{D} ^h[f]\big\|_{L_\pi(\mathcal{D})}$$ and $\mathbf{D} ^h [f](x)=f(x-h)-f(x)$. Equivalently we can define Besov space in term of wavelet coefficients (see H\"ardle \textit{et. al.} \cite{KerkPicTsyb} p. 123), $f$ belongs to the Besov space $\mathcal{B}^s_{{\pi} \infty}(\mathcal{D})$ if the quantity
\begin{align}
&\underset{j\geq-1}{\sup}2^{j(s+1/2-1/\pi)}\Big(\sum_{\lambda\in\Lambda_j}|\langle f,\psi_\lambda\rangle|^\pi\Big)^{1/\pi}\nonumber
\end{align} is finite, with usual modifications if $\pi=\infty$.

We need additional properties on the wavelet basis $\big(\psi_{\lambda}\big)_\lambda$, which are listed in the following assumption.
\begin{assumption}\label{Ass} For $p\geq1$,
\begin{itemize}
\item We have for some $\mathfrak{C}\geq 1$ $$\mathfrak{C}^{-1}2^{|\lambda|(p/2-1)}\leq \|\psi_\lambda\|_{L_p(\mathcal{D})}^p\leq \mathfrak{C} 2^{|\lambda|(p/2-1)}.$$
\item For some $\mathfrak{C}>0$, $\sigma >0$ and for all $s\leq\sigma$, $J\geq0$, we have \begin{align}\label{eq ass1} \big\|f-\sum_{j\leq J}\sum_{\lambda\in\Lambda_j} \langle f,\psi_\lambda\rangle\psi_\lambda\big\|_{L_p(\mathcal{D})}\leq \mathfrak{C}2^{-Js}\|f\|_{\mathcal{B}^s_{{\pi} \infty}(\mathcal{D})}.\end{align}
\item If $p\geq 1$, for some $\mathfrak{C}\geq 1$ and for any sequence of coefficients $\big(u_\lambda\big)_{\lambda\in\Lambda}$, \begin{align}\label{eq ass2}\mathfrak{C}^{-1}\Big\|\sum_{\lambda\in\Lambda}u_\lambda\psi_\lambda\Big\|_{L_p(\mathcal{D})}\leq \Big\|\Big( \sum_{\lambda\in\Lambda}|u_\lambda\psi_\lambda|^2\Big)^{1/2}\Big\|_{L_p(\mathcal{D})} \leq\mathfrak{C}\Big\|\sum_{\lambda\in\Lambda}u_\lambda\psi_\lambda\Big\|_{L_p(\mathcal{D})}.\end{align}
\item For any subset $\Lambda_0\subset\Lambda$ and for some $\mathfrak{C}\geq 1$ \begin{align}\label{eq ass3} \mathfrak{C}^{-1} \sum_{\lambda\in\Lambda_0}\|\psi_\lambda\|_{L_p(\mathcal{D})}^p\leq \int_{\mathcal{D}}\Big(\sum_{\lambda\in\Lambda_0}|\psi_\lambda(x)|^2\Big)^{p/2}\leq \mathfrak{C}\sum_{\lambda\in\Lambda_0}\|\psi_\lambda\|_{L_p(\mathcal{D})}^p .\end{align}
\end{itemize}
\end{assumption}

Property \eqref{eq ass1} ensures that definition \eqref{eq besov norm} of Besov spaces matches the definition in terms of linear approximation. Property \eqref{eq ass2} ensures that $\big(\psi_{\lambda}\big)_\lambda$ is an unconditional basis of $L_p$ and \eqref{eq ass3} is a super-concentration inequality (see Kerkyacharian and Picard \cite{KP00} p. 304 and p. 306).

\subsubsection*{Wavelet threshold estimator} Let $(\phi,\psi)$ be a pair of scaling function and mother wavelet that generate a basis $\big(\psi_{\lambda}\big)_\lambda$ satisfying Assumption \ref{Ass} for some $\sigma>0$. We rewrite \eqref{eq fdecomp1}
\begin{align*}
f&=\sum_{k\in \Lambda_0}\alpha_{0k}\phi_{0k}+\sum_{j\geq 1}\sum_{k\in\Lambda_j}\beta_{jk}\psi_{jk},
\end{align*} where $\phi_{0k}(\bullet)=\phi(\bullet-k)$ and $\psi_{jk}(\bullet)=2^{j/2}\psi(2^j\bullet-k)$ and \begin{align*} \alpha_{0k}&=\int \phi_{0k}(x)f(x)dx\\ \beta_{jk}&=\int \psi_{jk}(x)f(x)dx.\end{align*} For every $j\geq 0$, the set $\Lambda_j$ has cardinality $2^j$ and incorporates boundary terms that we choose not to distinguish in the notation for simplicity. An estimator of a function $f$ is obtained when replacing the $(\alpha_{0k})$ and $(\beta_{jk})$ by estimated values. In the sequel we uses $(\gamma_{jk})$ to design either $(\alpha_{0k})$ or $(\beta_{jk})$ and $(g_{jk})$ for the wavelet functions $(\phi_{0k})$ or $(\psi_{jk})$.

We consider classical hard threshold estimators of the form
\begin{align*}
\widehat{f}(\bullet)&= \sum_{k\in \Lambda_0}\widehat{\alpha_{0k}}\phi_{0k}(\bullet)+\sum_{j= 1}^J\sum_{k\in\Lambda_j}\widehat{\beta_{jk}}\mathds{1}_{\big\{|\widehat{\beta_{jk}}|\geq\eta\big\}}\psi_{jk}(\bullet),\end{align*} where $\widehat{\alpha_{0k}}$ and $\widehat{\beta_{jk}}$ are estimators of $\alpha_{0k}$ and $\beta_{jk}$, $J$ and $\eta$ are respectively the resolution level and the threshold, possibly depending on the data. Thus to construct $\widehat{f}$ we have to specify estimators $(\widehat{\gamma_{jk}})$ of the $(\gamma_{jk})$ and the coefficients $J$ and $\eta$.

\subsection{Construction of the estimator}

Assume that we have $\lfloor T\Delta^{-1}\rfloor$ discrete data at times $i\Delta$ for some $\Delta >0$ of the process $X$
\begin{equation*}
\big(X_\Delta,\ldots,X_{\lfloor T\Delta^{-1}\rfloor\Delta}\big).
\end{equation*} Introduce the increments
\begin{align*}
\mathbf{D}^\Delta X_i=X_{i\Delta}-X_{(i-1)\Delta},\ \ \ \mbox{ for } i=1,\dots,\lfloor T\Delta^{-1}\rfloor,\end{align*}
 where $X_0=0$. By Assumption \ref{ass f}, they are identically distributed but not independent. \begin{proposition}\label{PropDefOperator}
The distribution of the increment $\mathbf{D}^\Delta X_{1}$ is $$\big(1-p(\Delta)\big)\delta_0+p(\Delta)\mathbf{P}_\Delta[f]$$ where $\delta_0$ is the dirac delta function, $p(\Delta)=\PP(R_\Delta\ne 0)$ and
\begin{align}\label{DefOperator}\mathbf{P}_\Delta[f]&=\sum_{m=1}^\infty p_m(\Delta) f^{\star m},
\end{align}
 where $\star $ is the convolution product, $f^{\star m}$ is $f$ convoluted $m$ times and $$p_m(\Delta)=\PP\big(R_\Delta=m|R_\Delta\ne0\big).$$
\end{proposition}
It is straightforward to verify that the operator $\mathbf{P}_\Delta$ is a mapping from $\mathcal{F}(\R)$ to itself. The following Lemma gives a polynomial control of the coefficients $\big(p_m(\Delta)\big)$. It is widely used in Sections \ref{section fast} and \ref{section slow} and does not depend on the rate at which $\Delta_T$ decays to 0.
\begin{lemma}\label{lem pmControl}
 Assume $\tau(0)>0$ and let $\Delta_0$ be such that $$\int_0^{\Delta_0}\tau(t)dt\leq\frac{1}{2}\ \ \ \mbox{ and } \ \ \ \underset{t\in[0,\Delta_0]}{\sup}\tau(t)\leq 2\tau(0).$$ For all $\Delta\leq\Delta_0$ we have $$1- 2\tau(0)\Delta\leq p_1(\Delta)\leq1,$$
 and for $m\geq2$  $$0\leq p_m(\Delta)\leq 2 \frac{\big(2\tau(0)\big)^{m-1}}{m!}\Delta^{m-1},$$ where the $\big(p_m(\Delta))$ are defined in Proposition \ref{PropDefOperator}.
 \end{lemma}

\begin{remark}
The assumption $\tau(0)>0$ in Lemma \ref{lem pmControl} ensures that the given inequalities are sharp. In the Poisson case it is always true since $\tau(0)$ is the positive intensity. In the renewal case we may have $\tau(0)=0$, if so two cases must be distinguished. The first one is when $\tau$ as infinitely many derivatives null at 0; it is the case if $\tau$ is bounded away from 0. Then straightforward computations give for any $K$ in $\N$: $p_1(\Delta)=1+O\big(\Delta^K\big),$ thus the procedure of Section \ref{section fast} enables to achieve optimal rates even in slow microscopic regimes. It is not the purpose of this paper. The second case is $\tau(0)=0$ but there exists $l_0$ in $\N$ such that $\tau^{(l_0)}(0)>0$, then Lemma \ref{lem pmControl} can be adapted replacing $\tau(0)$ by $\tau^{(l_0)}(0)$ and $\Delta$ by $\Delta^{l_0}$. In the sequel we assume that $\tau(0)>0$ and leave to the reader the changes to be made when $\tau(0)=0$.
\end{remark}

In this Section we consider the regimes for which $\Delta=\Delta_T$ is such that $T\Delta_T^2=O(1),$ up to logarithmic factors in $T$. To estimate $f$, we use the approximation $\mathbf{P}_{\Delta_T}[f]\approx f.$ It is equivalent to consider that nonzero increments are realisations of $f$. We construct wavelet threshold density estimators of ${\mathbf{P}_\Delta}[f]$ from the observations $$\big(\mathbf{D}^\Delta X_i,i=1,\ldots,\lfloor T\Delta^{-1}\rfloor\big).$$ Define the wavelet coefficients \begin{align}\label{eq est coeff NC}
\widehat{\gamma}_{jk}&=\frac{1}{\big(1-p(\Delta)\big)\lfloor T\Delta^{-1}\rfloor}\sum_{i=1}^{\lfloor T\Delta^{-1}\rfloor}g_{jk} \Big(\mathbf{D}^{\Delta} X_i\Big)\mathds{1}_{\big\{\mathbf{D}^{\Delta} X_i\ne 0\big\}},\end{align} where $p(\Delta)$ is defined in Proposition \ref{PropDefOperator}. Let $\eta>0$ and $J\in \N \setminus\{0\}$, the estimator $\widehat{P_{\Delta}}$ of $\mathbf{P}_\Delta[f]$ is for $x$ in $\mathcal{D}$ \begin{align}\label{eq est P NC}
\widehat{P_{\Delta}}(x)&=\sum_{k}\widehat{\alpha}_{0k}\phi_{0k}(x)+\sum_{j= 0}^J\sum_{k}\widehat{\beta}_{jk}\mathds{1}_{\big\{|\widehat{\beta}_{jk}|\geq \eta\big\}}\psi_{jk}(x).
\end{align}

\begin{definition} \label{def est f NC}We define $\widehat{f}_{T,\Delta}$ an estimator of $f$ for $x$ in $\mathcal{D}$ as
 \begin{align}\label{eq Est f Threshold NC}
\widehat{f}_{T,\Delta}(x)&=\widehat{P_{\Delta}}(x).\end{align}
\end{definition}

\subsection{Convergence rates}
We estimate densities $f$ which verify a smoothness property in term of Besov balls $$\mathcal{F}(s,{\pi},\mathfrak{M})=\big\{ f\in \mathcal{F}(\R), \|f\|_{\mathcal{B}^s_{{\pi} \infty}(\mathcal{D})}\leq\mathfrak{M}\big\},$$ where $\mathfrak{M}$ is a positive constant. We are interested in estimating $f$ on the compact interval $\mathcal{D}$, that is why we only impose that its restriction to $\mathcal{D}$ belongs to a Besov ball.

\begin{theorem}  \label{thm Renewal 1} We work under Assumptions \ref{ass f} and \ref{Ass}, let $\Delta_T$ be such that $T\Delta_T^2=O(1)$ up to logarithmic factors in $T$. Let $\pi>0$, $\sigma>s>1/\pi$, $p\geq1\wedge \pi$ and $\widehat{P_{\Delta_T}}$ be the wavelet threshold estimator of $\mathbf{P}_{\Delta_T}[f]$ on $\mathcal{D}$ constructed from $(\phi,\psi)$ and defined in \eqref{eq est P NC}. Take $J$ such that $$2^JT^{-1}\log\big(T^{1/2}\big)\leq 1,$$ and $$\eta=\kappa T^{-1/2}\sqrt{\log\big(T^{1/2}\big)},$$ for some $\kappa>0$.
Let
\begin{align}\label{eq alpha}\alpha(s,p,\pi)=\min\Big\{\frac{s}{2s+1},\frac{s+1/p-1/{\pi}}{2\big(s+1/2-1/{\pi}\big)}\Big\}.\end{align}
1)  The estimator $\widehat{P_{\Delta_T}}$ verifies for large enough $T$ and sufficiently large $\kappa>0$ \begin{align*}
\underset{\mathbf{P}_{\Delta_T}[f]\in\mathcal{F}(s,{\pi},\mathfrak{M})}{\sup}\big(\E \big[\big\|\widehat{P_{\Delta_T}} -\mathbf{P}_{\Delta_T}[f] \big\|_{L_p(\mathcal{D})}^p\big] \big)^{1/p}&\leq \mathfrak{C}  T^{-\alpha(s,p,\pi)},
\end{align*} up to logarithmic factors in $T$ and where $\mathfrak{C}$ depends on $s,\pi,p,\mathfrak{M},\phi,\psi,\mu$.\\
2)The estimator $\widehat{f}_{T,\Delta_T}$ defined in \eqref{eq Est f Threshold NC} verifies for $T$ large enough, sufficiently large $\kappa>0$ and any positive constants $\underline{\mathfrak{a}}<\overline{\mathfrak{a}}$ \begin{align*}
\underset{(\mu,\tau(0))\in[\underline{\mathfrak{a}},\overline{\mathfrak{a}}]^2}{\sup}\ \ \underset{f\in\mathcal{F}(s,{\pi},\mathfrak{M})}{\sup}\big(\E \big[ \| \widehat{f}_{T,\Delta_T}-f\|_{L_p(\mathcal{D})}^p\big]\big)^{1/p}&\leq \mathfrak{C} T^{-\alpha(s,p,\pi)},
\end{align*} up to logarithmic factors in $T$, where $\mu=\int t\tau(t)dt$ and where $\mathfrak{C}$ depends on $s,\pi,p,\mathfrak{M},\phi,\psi,\underline{\mathfrak{a}}$ and $\overline{\mathfrak{a}}$.

\end{theorem} \noindent The proof of Theorem \ref{thm Renewal 1} is postponed to Section \ref{proof thm 1}. Theorem \ref{thm Renewal 1} guarantees that when $\Delta=\Delta_T$ tends rapidly to 0, namely $T\Delta_T^2=O(1)$, the approximation $f\approx \mathbf{P}_{\Delta_T}[f]$ enables to achieve minimax rates of convergence (see Section \ref{section discuss}). The estimator does not depend on $\tau$.

\section{Estimation of $f$ in the slow microscopic regime\label{section slow}}

In this Section we consider the regimes for which there exists $0<\delta< 1$ with $T\Delta_T^2=O(T^\delta),$ up to logarithmic factors in $T$.

\subsection{Construction of the estimator}
We construct the estimator corrected at order $K$, following the estimation procedure described in Section \ref{section our res}.

\subsubsection*{Construction of the inverse}
Define the space \begin{align*}
\mathcal{H}(s,{\pi},\mathfrak{O},\mathfrak{N})=\Big\{h, \|h\|_{L_1(\mathcal{D})}\leq \mathfrak{O}, \|h\|_{\mathcal{B}^s_{{\pi} \infty}(\mathcal{D})}\leq\mathfrak{N}\Big\},
\end{align*} where $\mathfrak{O}$ is any constant strictly greater than 1 and $\mathfrak{N}$ is a positive constant strictly greater than $\mathfrak{M}$. The space $\mathcal{H}(s,{\pi},\mathfrak{O},\mathfrak{N})$ is a subset of $\mathcal{B}^s_{{\pi} \infty}(\mathcal{D})$ which is a Banach space if equipped with the Besov norm \eqref{eq besov norm}.

First we approach the inverse of $\mathbf{P}_\Delta$ with a fixed point method. Consider the mapping $\mathbf{H}_{\Delta,f}$ defined for $h$ in $\mathcal{H}(s,{\pi},\mathfrak{O},\mathfrak{N})$ by
\begin{align}\label{eq Op Point fixe}
\mathbf{H}_{\Delta,f}[h]:=  \mathbf{P}_\Delta[f]+h-\mathbf{P}_\Delta[h]. \end{align} We immediately verify that $f$ is a fixed point: $\mathbf{H}_{\Delta,f}[f]=f$. The constraints $1<\mathfrak{O}$ and $\mathfrak{M}<\mathfrak{N}$ ensure that if $f$ is in $\mathcal{F}(s,{\pi},\mathfrak{M})$, then $\mathbf{H}_{\Delta,f}[h]$ sends elements of $\mathcal{H}(s,{\pi},\mathfrak{O},\mathfrak{N})$ into itself (see Proposition \ref{prop complet contract}). The following Proposition guarantee that the definition of the operator \eqref{eq Op Point fixe} matches the assumptions of the Banach fixed point theorem.
\begin{proposition}\label{prop complet contract} The following properties hold.\\
 1) Let $\pi\geq 1$, the space $\big(\mathcal{H}(s,{\pi},\mathfrak{O},\mathfrak{N}),\|.\|_{\mathcal{B}^s_{{\pi} \infty}(\mathcal{D})}\big)$ is a closed set of a Banach space and is then complete.\\
2) The mapping $\mathbf{H}_{\Delta,f}$ sends elements of $\mathcal{H}(s,{\pi},\mathfrak{O},\mathfrak{N})$ into itself and is a contraction. For all $h_1,h_2 \in \mathcal{H}(s,{\pi},\mathfrak{O})$ we have that
\begin{align*}
\big\|\mathbf{H}_{\Delta,f}[h_1]-\mathbf{H}_{\Delta,f}[h_2]\big\|_{\mathcal{B}^s_{{\pi} \infty}(\mathcal{D})}&\leq \mathfrak{K}(\Delta)\|h_1-h_2\|_{\mathcal{B}^s_{{\pi} \infty}(\mathcal{D})},
\end{align*} where \begin{align}\label{eq contract}  \mathfrak{K}(\Delta)=
2\mathfrak{O}(e^{2\tau(0)\Delta}-1)+2\tau(0)\Delta.
\end{align} Moreover since $\Delta_T\rightarrow0$ we have \begin{align}\label{eq contract 0}\mathfrak{K}(\Delta_T)\leq \mathfrak{C}\Delta_T<1\end{align}for some positive constant $\mathfrak{C}$ depending on $\tau(0)$ and $\mathfrak{O}$. \end{proposition}

\noindent Proposition \ref{prop complet contract} enables to apply the Banach fixed point theorem; we derive that $f$ is the unique fixed point of $\mathbf{H}_{\Delta,f}$ and from any initial point $h_0$ in $\mathcal{H}(s,{\pi},\mathfrak{O},\mathfrak{N})$ we have $$\big\|f-\mathbf{H}_{\Delta,f}^{\circ K} [h_0]\big\|_{\mathcal{B}^s_{{\pi} \infty}(\mathcal{D})}\rightarrow 0\ \ \ \mbox{ as } \ K\rightarrow\infty,$$ where $\circ$ stands for the composition product and $\mathbf{H}_{\Delta,f}^{\circ K}$ is $\mathbf{H}_{\Delta,f}\circ \ldots\circ \mathbf{H}_{\Delta,f}$, $K$ times. We choose $h_0=\mathbf{P}_\Delta[f]$ as a starting point (Lemma \ref{lem loicontY boule besov} in Section \ref{Section proof} ensures that $\mathbf{P}_\Delta[f]$ belongs to $\mathcal{H}(s,{\pi},\mathfrak{O},\mathfrak{N})$).
\begin{proposition}\label{prop inverseTronqueeK}
Let $\pi\geq 1$ and define the operator $\mathbf{L}_{\Delta,K}$ as the $K$th degree Taylor polynomial of $\mathbf{H}_{\Delta,f}^{\circ K}\big[\mathbf{P}_\Delta[f]\big]$ in $\Delta$. It verifies for $p\geq1$\begin{align}\label{eq Taylor Approx renewal}\Big\|\mathbf{H}_{\Delta,f}^{\circ K
}\big[\mathbf{P}_\Delta[f]\big]-\mathbf{L}_{\Delta,K} \Big\|_{L_p(\mathcal{D})}\leq\mathfrak{C}\Delta^{K+1}\end{align} where $\mathfrak{C}$ is a positive constant depending on $\tau(0)$, $\mathfrak{M}$ and $\mathfrak{O}$. Moreover we have
\begin{align}\label{eq LfK linear}
\mathbf{L}_{\Delta,K}=\sum_{m=1}^{K+1} l_m(\Delta) {\mathbf{P}_\Delta[f]}^{\star m},
\end{align} where for $m=1,\ldots,K+1$ we have $|l_m(\Delta)|\leq \mathfrak{C}\Delta^{m-1}$ where $\mathfrak{C}$ is a positive constant that depends on $\tau(0)$ and $K$.
\end{proposition}

\subsubsection*{Construction of estimators of the $\mathbf{P}_\Delta[f]^{\star m}$}
Consider the increments $\big(\mathbf{D}^\Delta X_i=X_{i\Delta}-X_{(i-1)\Delta},i=1,\dots,\lfloor T\Delta^{-1}\rfloor\big)$
introduced earlier and define the nonzero ones using
\begin{align*}S_1&=\inf\big\{j, \mathbf{D}^\Delta X_j\ne0\big\}\wedge \lfloor T\Delta^{-1}\rfloor\\S_i&=\inf\big\{j>S_{i-1}, \mathbf{D}^\Delta X_j\ne0\big\}\wedge \lfloor T\Delta^{-1}\rfloor\ \ \ \mbox{for }i\geq1,\end{align*} where $S_i$ is the random index of the $i$th jump. Let
$$N_T=\sum_{i=1}^{\lfloor T\Delta^{-1}\rfloor}\mathds{1}_{\{\mathbf{D}^\Delta X_i\ne0\}}
$$ the random number of nonzero increments observed over $[0,T]$. By Assumption \ref{ass f}, on the event $\{\mathbf{D}^\Delta X_i=0\},$ no jump occurred between $(i-1)\Delta$ and $i\Delta$. In the microscopic regime when $\Delta=\Delta_T\rightarrow 0$ as $T$ goes to infinity many increments are null and convey no information about $f$, hence for the estimation of $f$ we focus on the nonzero ones $$\big(\mathbf{D}^\Delta X_{S_1},\ldots,\mathbf{D}^\Delta X_{S_{N_T}}\big).$$ They are identically distributed of density given by \eqref{DefOperator}; Lemma \ref{lem pmControl} still applies.

We construct wavelet threshold density estimators of the $K+1$ first convolution powers of $\mathbf{P}_\Delta[f]$; define the wavelet coefficients for $m\geq 1$
\begin{align}\label{eq est Coeffconvol}
\widehat{\gamma}^{(m)}_{jk}&=\frac{1}{N_{T,m}}\sum_{i=1}^{N_{T,m}}g_{jk} \Big(\mathbf{D}^{\Delta}_m X_{S_i}\Big),\end{align} where $N_{T,m}=\big\lfloor N_T/m \big\rfloor\geq 1$ for large enough $T$ and $$\mathbf{D}^{\Delta}_m X_{S_i}=\mathbf{D}^\Delta X_{S_i}+ \mathbf{D}^\Delta X_{S_{N_{T,m}+i}}+\dots+\mathbf{D}^\Delta X_{S_{(m-1)N_{T,m}+i}}.$$ Let $\eta>0$ and $J\in \N \setminus\{0\},$ define $\widehat{P_{\Delta,m}}$ the estimator of $\mathbf{P}_\Delta[f]^{\star m}$ over $\mathcal{D}$ for $m\geq 1$ \begin{align}\label{eq est convol}
\widehat{P_{\Delta,m}}(x)&=\sum_{k}\widehat{\alpha}_{0k}^{(m)}\phi_{0k}(x)+\sum_{j= 0}^J\sum_{k}\widehat{\beta}_{jk}^{(m)}\mathds{1}_{\big\{|\widehat{\beta}_{jk}^{(m)}|\geq \eta\big\}}\psi_{jk}(x),\ \ \ x\in\mathcal{D}.
\end{align}

As mentioned earlier $\tau$ is a nuisance that needs to be estimated. To simplify the problem, we make the following parametric assumption on $\tau$.
\begin{assumption}\label{ass param tau}
Assume there exists $\vartheta$ in $\Theta$ a compact subset of $\R$ such that \begin{align*}
\tau(x)&=\tau_1(x,\vartheta),\ \ \ \ \forall x\in [0,\infty),
\end{align*} where $\tau_1$ is known, $\tau_1(0,\vartheta)>0$ and $\vartheta\rightarrow\tau_1(.,\vartheta)$ is $C^1$ .
Assume there exists $q$ from $\Theta$ to $[0,1]$, invertible, such that $q(\vartheta)=\PP(R_\Delta\ne0)$ and whose inverse $q^{-1}$ is bounded.\end{assumption}
Assumption \ref{ass param tau} enables to estimate the unknown coefficients $\big(p_m(\Delta)\big)$ and $\big(l_m(\Delta)\big)$, and to compute the estimator of $f$ defined hereafter.

\begin{definition} \label{def est corr K}Let $\widehat{f}^K_{T,\Delta}$ be the estimator corrected at order $K$ defined for $K$ in $\N$ and $x$ in $\mathcal{D}$ as
 \begin{align}\label{eq Est f Threshold Corr renew}
\widehat{f}^K_{T,\Delta}(x)&=\sum_{m=1}^{K+1} l_m(\Delta,\widehat{\vartheta_T}) \widehat{P_{\Delta, m}}(x),\end{align} where $$\widehat{\vartheta_T}=q^{-1}\Big(\frac{1}{\lfloor T \Delta^{-1}\rfloor}\sum_{i=1}^{\lfloor T \Delta^{-1}\rfloor}\mathds{1}_{\mathbf{D}^\Delta X_i\ne0}\Big)$$ and the $l_m(\Delta,\vartheta)$ are defined in Proposition \ref{prop inverseTronqueeK}.
\end{definition}

\noindent When $\Delta_T$ satisfies $T\Delta_T^2=O(1)$, $\widehat{f}^0_{T,\Delta}$ defined in \eqref{eq Est f Threshold Corr renew} with $K=0$ and $\widehat{f}_{T,\Delta}$ defined in \eqref{eq Est f Threshold NC} coincides.

\subsection{Convergence rates}
\begin{assumption}\label{ass queue tau}
Assume that there exist $(\mathfrak{A},\mathfrak{a},\mathfrak{g})$ positive constants such that \begin{align*}
\tau(x)\leq \mathfrak{A}\exp\big(-\mathfrak{a}x^\mathfrak{g}),\ \ \ \ \forall x\in [0,\infty).
\end{align*}
\end{assumption}Assumption \ref{ass queue tau} is a technical condition which ensures that $\tau$ has moments of all order. It is used in the proofs to replace $N_T/T$ by its asymptotic deterministic limit. Compactly supported densities and densities with subexponential queues satisfies Assumption \ref{ass queue tau}.

\begin{theorem}  \label{thm Renewal 2} We work under Assumptions \ref{ass f}, \ref{Ass}, \ref{ass param tau} and \ref{ass queue tau} and assume that there exists $0<\delta< 1$ such that $$T\Delta_T^2=O(T^\delta),$$ up to logarithmic factors in $T$. Let $\pi\geq 1$, $\sigma>s>1/\pi$, $p\geq1$ and $\widehat{P_{\Delta_T,m}}$ be the threshold wavelet estimator of $\mathbf{P}_{\Delta_T}[f]^{\star m}$ on $\mathcal{D}$ constructed from $(\phi,\psi)$ and defined in \eqref{eq est convol}. Take $J$ such that $$2^JT^{-1}\log\big(T^{1/2}\big)\leq 1,$$ and $$\eta=\kappa T^{-1/2}\sqrt{\log\big(T^{1/2}\big)},$$ for some $\kappa>0$.\\
1) For $m\geq1$ the estimator $\widehat{P_{\Delta_T,m}}$ of $\mathbf{P}_{\Delta_T}[f]^{\star m}$ verifies for sufficiently large $\kappa>0$ \begin{align*}
\underset{\mathbf{P}_{\Delta_T}[f]^{\star m}\in\mathcal{F}(s,{\pi},\mathfrak{M})}{\sup}\big(\E \big[\big\|\widehat{P_{\Delta_T,m}} -\mathbf{P}_{\Delta_T}[f]^{\star m}\big\|_{L_p(\mathcal{D})}^p\big] \big)^{1/p}&\leq \mathfrak{C}  T^{-\alpha(s,p,\pi)},
\end{align*} up to logarithmic factors in $T$, where $\alpha(s,p,\pi)$ is defined in \eqref{eq alpha} and where $\mathfrak{C}$ depends on $s,\pi,p,\mathfrak{M},\phi,\psi$ and $\vartheta$.\\
2) The estimator corrected at order $K$ $\widehat{f}^K_{T,\Delta}$ for $K\in\N$ defined in \eqref{eq Est f Threshold Corr renew} verifies for $T$ large enough, sufficiently large $\kappa>0$ and any compact set $\Theta\subset\R$ \begin{align*}
\underset{\vartheta\in\Theta}{\sup}\underset{f\in\mathcal{F}(s,{\pi},\mathfrak{M})}{\sup}\big(\E \big[ \| \widehat{f}^K_{T,\Delta} -f\|_{L_p(\mathcal{D})}^p\big]\big)^{1/p}&\leq \mathfrak{C} \max\big(T^{-\alpha(s,p,\pi)},\Delta_T^{K+1}\big),
\end{align*}up to logarithmic factors in $T$ and where $\mathfrak{C}$ depends on $s,\pi,p,\mathfrak{M},\phi,\psi$ and $K$.
\end{theorem}
\noindent The proof of Theorem \ref{thm Renewal 2} is postponed to Section \ref{section proof thm2}. Since $\alpha(s,p,\pi)\leq1/2$, Theorem \ref{thm Renewal 2} ensures that whenever $\Delta_T$ and $T$ are polynomially related it is always possible to find $K_0$ such that the estimator corrected at order $K_0$ achieves the minimax rate of convergence (see Section \ref{section discuss}). If $\Delta_T$ decays slower than any power of $1/T$, for instance if it decreases logarithmically with $T$, the estimator corrected at order $K$ still provide a consistent estimator of $f$.

\section{A numerical example\label{section num ex}}
In this Section we illustrate the results of Theorems \ref{thm Renewal 1} and \ref{thm Renewal 2}. In both cases we compare the performances of our estimator with an oracle: the wavelet estimator we would compute in the idealised framework where all the jumps are observed
\begin{align*}
\widehat{f}^{Oracle}(x)=\sum_{k}\widehat{\alpha}_{0k}^{Oracle}\phi_{0k}(x)+\sum_{j= 0}^J\sum_{k}\widehat{\beta}_{jk}^{Oracle}\mathds{1}_{\big\{|\widehat{\beta}_{jk}^{Oracle}|\geq \eta\big\}}\psi_{jk}(x),\end{align*} where $$\widehat{\alpha}_{0k}^{Oracle}=\frac{1}{R_T}\sum_{i=1}^{R_T}\phi_{0k}(\xi_i)\ \ \ \mbox{and}\ \ \ \widehat{\beta}_{jk}^{Oracle}=\frac{1}{R_T}\sum_{i=1}^{R_T}\phi_{0k}(\xi_i),$$ $R_T$ being the value of the renewal process $R$ at time $T$ and $(\xi_i)$ the jumps. The parameters $J$ and $\eta$ as well as the wavelet bases $(\phi,\psi)$ are the same for all the estimators.

We consider a renewal process with a $Beta(1,\vartheta)$ interarrival density $\tau$. We have $\vartheta=3$, the first shape parameter is set to 1 to ensure the condition $0<\tau_1(0,\vartheta)<\infty$. We estimate the compound law given by $$f(x)=(1-a) f_1(x)+af_2(x),$$ where $f_1$ is the uniform distribution over $[-2,2]$ and $f_2$ is a Laplace with location parameter 1 and scale parameter 0.5, we take $a=0.5$. We estimate the mixture $f$ on $\mathcal{D}=[-10,10]$ with the estimator corrected at order $K$ for different values of $K$ and study the results with the $L_2$ error. We also compare them with the oracle $\widehat{f}^{Oracle}$. Wavelet estimators are based on the evaluation of the first wavelet coefficients, to perform those we use Symlets 4 wavelet functions and a resolution level $J=10$. Moreover we transform the data in an equispaced signal on a grid of length $2^L$ with $L=8$, it is the binning procedure (see H\"{a}rdle \textit{et al.} \cite{KerkPicTsyb} Chap. 12). The threshold is chosen as in Theorems \ref{thm Renewal 1} and \ref{thm Renewal 2}. The estimators we obtain take the form of a vector giving the estimated values of the density $f$ on the uniform grid $[-10,10]$ with mesh $0.01$. We use the wavelet toolbox of \textsf{Matlab}.

\subsection{Illustration in the fast microscopic case}
In this case we choose $\Delta=T^{-1/2}$. Figure \ref{Fig Fast1} represents the estimator $\widehat{f}_{T,\Delta}$ of Definition \ref{def est f NC} and the oracle. The estimators are evaluated on the same trajectory. They are quite hard to distinguish, what is confirmed by the comparison of their $L_2$ losses.

\begin{figure}[H]
\begin{center}
\includegraphics[scale=0.4]{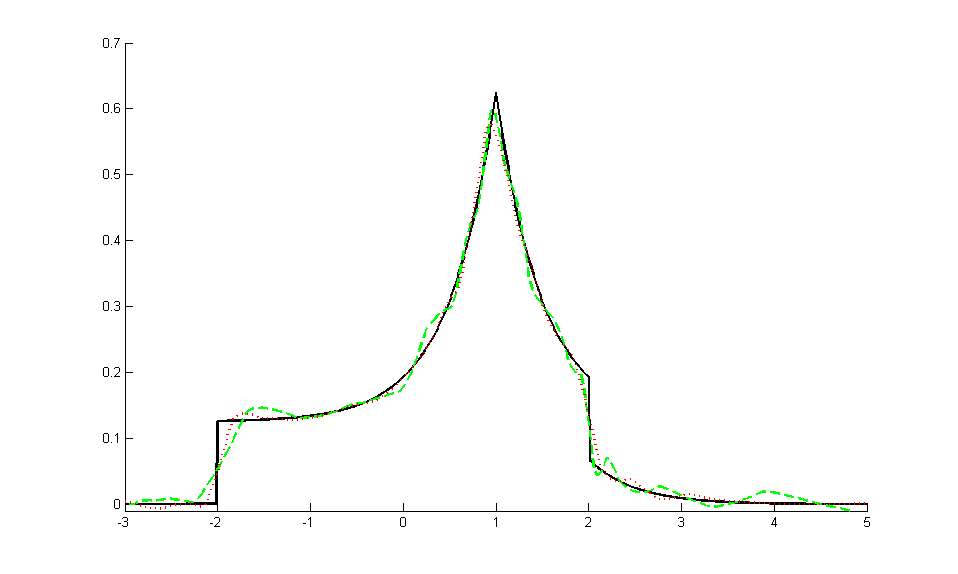}\hfill
\caption{Estimators of the density $f$ (plain dark) for $T=10000$ and $\Delta=0.01$: the oracle (dotted red) and the estimator $\widehat{f}_{T,\Delta}$ (dashed green). }\label{Fig Fast1}
\end{center}
\end{figure}
We approximate the $L_2$ errors by Monte Carlo. For that we compute $M=1000$ times each estimator (for $T=10000$ and $\Delta=0.01$) and approximate the $L_2$ loss by$$\frac{1}{M}\sum_{i=1}^M\Big(\sum_{p=0}^{2000} \big(\widehat{f}(-10+0.01 p)-f(-10+0.01 p)\big)^2\times0.01\Big).$$ For each Monte Carlo iteration the estimators are evaluated on the same trajectory. The results are reproduced in the following table. \begin{center}
\begin{tabular}{|l|c|c|c|c|c|}
  \hline
  Estimator & Oracle & $\widehat{f}_{T,\Delta}$\\
  \hline
  $L_2$ error \small{($\times 10^{-4}$)} &  0.1916  &  0.2040     \\
  \hline
  Standard deviation \small{($\times 10^{-5}$)} &    0.4519   & 0.4605     \\
  \hline
\end{tabular}
\end{center}

\subsection{Illustration in the slow microscopic case}

We now study the behaviour of the estimator corrected at order $K$ for different values of $K$. We choose $T=10000$ and $\Delta=0.1$. In that case $T\Delta^2$ is large but $T\Delta^4$ is 1. According to Theorem \ref{thm Renewal 2} we should observe that the estimator corrected at order 2, behaves as the oracle. Figure \ref{Fig Slow1} represents the estimators $\widehat{f}^K_{T,\Delta}$ defined in Definition \ref{def est corr K} for $K\in \{0,1,2,3\}$ and the oracle. The estimators are evaluated on the same trajectory. They all manage to reproduce the shape of the density $f$, and graphically apart from the estimator corrected at order 0 they are difficult to distinguish.

\begin{figure}[H]
\begin{center}
\includegraphics[scale=0.4]{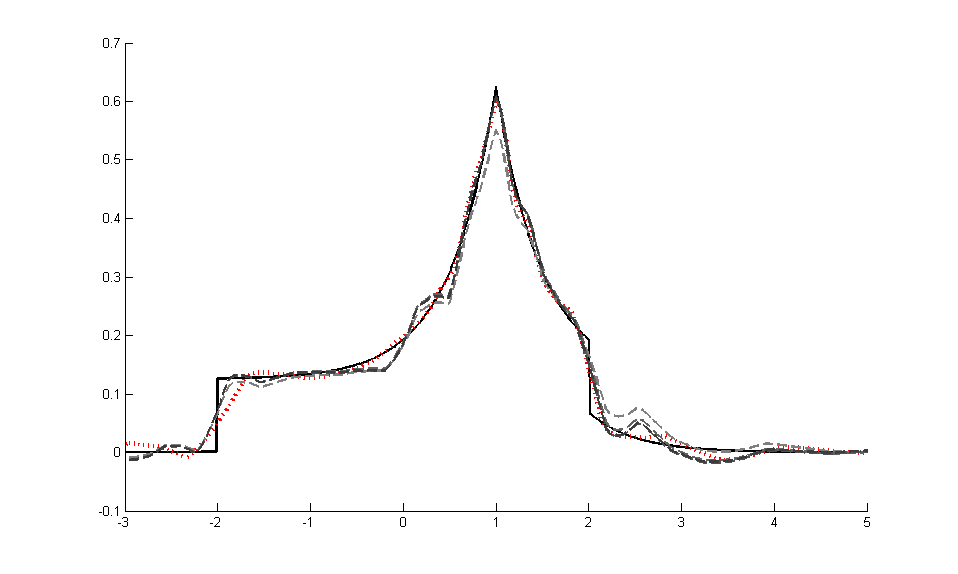}\hfill
\caption{Estimators of the density $f$ (plain dark) for $T=10000$ and $\Delta=0.1$: the oracle (dotted red) and the estimator $\widehat{f}^K_{T,\Delta}$ for $K=0,1,2,3$ (dashed light to dark grey). }\label{Fig Slow1}
\end{center}
\end{figure}
We compare their $L_2$ losses in the following tabular. \begin{center}
\begin{tabular}{|l|c|c|c|c|c|}
  \hline
  Estimator & Oracle & $K=0$ & $K=1$ & $K=2$ & $K=3$\\
  \hline
  $L_2$ error \small{($\times 10^{-4}$)} &  0.1896  &  0.5176  &  0.3037 &   0.2959   & 0.2946   \\
  \hline
  Standard deviation \small{($\times 10^{-5}$)} &    0.4348  &  0.7800  &  0.7533   & 0.7462 &   0.7466   \\
  \hline
\end{tabular}
\end{center} This confirms that there is an actual gain in considering the estimator corrected at order 1 instead of the uncorrected one. In the following table we estimate the $\big(p_m(\Delta)\big)$ defined in Proposition \ref{PropDefOperator}. \begin{center}
\begin{tabular}{|l|c|c|c|}
  \hline
  Estimated quantity & $\widehat{p_1}$ & $\widehat{p_2}$ & $\widehat{p_3}$\\
  \hline
  Estimation &  0.8527   & 0.1327  &  0.0135   \\
  \hline
  Standard deviation \small{($\times 10^{-3}$)}  &   0.9185  &  0.7388  &  0.1597   \\
  \hline
\end{tabular}
\end{center} It turns out that making no correction is equivalent to estimate a density on a data set where $15\%$ of the observations are realisations of a law which is not target. This explains why it is relevant to take them into account when estimating $f$. Considering more than 1 or 2 corrections is unnecessary as the $L_2$ losses get stable afterwards. The $L_2$ loss of the oracle is strictly lower than the loss of the estimator corrected at order $K$, even for large $K$. That difference is explained by the fact that to estimate the $m$th convolution power we do not use $N_T$ data points but $N_{T,m}=\lfloor N_T/m\rfloor$. Therefore we do not loose in terms of rate of convergence, but we surely deteriorate the constants in comparison with the oracle.

\section{Discussion and Conclusion\label{section discuss}}

\paragraph{Attainable rates.}
Without loss of generality, assuming $T$ is an integer if we observe $T$ independent realisations of the density $f$, it is possible to achieve the minimax rates of convergence $T^{-\alpha(s,\pi,p)}$ (see for instance Donoho \textit{et al.} \cite{Donoho96}). When the process $X$ is continuously observed over $[0,T]$, we have $R_T$ independent and identically distributed realisations of $f$. Moreover for $T$ large enough, the elementary renewal theorem guarantees that $R_T$ is of the order of $T$ (see for instance Lindvall \cite{Lindvall}). It follows that the estimators of $f$ given in Sections \ref{section fast} and \ref{section slow} enables to attain the minimax rates of convergence of an experiment where $X$ is continuously observed.

\paragraph{Comparison with a previous work.} The results of this paper are the generalisation to the renewal reward case of Duval \cite{Du2}; a compound Poisson process is a particular renewal reward process and Theorems \ref{thm Renewal 1} and \ref{thm Renewal 2} enable to recover the results of \cite{Du2}. However in this paper we do not have an explicit formula for the estimator corrected at order $K$ but only  a construction method. In the Poisson case it is much more simpler to apply the results of \cite{Du2}.

\paragraph{Extension to the case where $\Delta$ is fixed.} We established Theorem \ref{thm Renewal 2} for $\Delta_T$ vanishing to 0. Since the approximation of the inverse depends only the fact that $\mathbf{H}_{\Delta,f}$ is a contraction, the method remains valid for $\Delta$'s such that $\mathfrak{K}(\Delta)$ defined in \eqref{eq contract} is strictly lower than 1. Which means that we can expand the results to cases where $\Delta$ does not go to 0 but satisfies $\mathfrak{K}(\Delta)<1.$ The value of the maximum value $\Delta_1$ satisfying the former inequality depends on $\tau(0)$ and $\mathfrak{O}$ but is not only determined by \eqref{eq contract}. Another hidden condition on $\Delta$ have to be satisfied for $\mathbf{H}_{\Delta,f}$ to send elements of $\mathcal{H}(s,{\pi},\mathfrak{O},\mathfrak{N})$ into itself. Then to find $\Delta_1$ one has to solve an optimisation program with constraints to find $\Delta_1$ and $\mathfrak{O}$ giving the maximum coverage for $\Delta$. To get an idea of the value of $\Delta_1$ we use the function \texttt{NMaximize} of \texttt{Mathematica} and find that one should take $\mathfrak{O}=1.645$ and $\Delta_1=0.071/\tau(0)>0,$ which is positive.
The results of Theorem \ref{thm Renewal 2} should generalise in for all $\Delta_T\rightarrow\Delta_\infty$ such that $\Delta_\infty<\Delta_1$ and for $K\in \N$ the rate of convergence for the estimator corrected at order $K$ is bounded by $$\max\big\{T^{-\alpha(s,p,\pi)},\Delta^{K+1}_\infty\big\}.$$ However to achieve suitable rates theoretically one should consider larger $K$, therefore the dependency in $K$ in the constants need to be handled carefully. In practice for $T=10000$ and $\Delta=0.1$ considering $K=2$ appears sufficient to have $T^{-\alpha(s,p,\pi)}$ predominant in front of $\Delta^{K+1}$.

\paragraph{Discussion on Assumptions \ref{ass f} and \ref{ass param tau}.} In the present paper we made two simplifying assumptions on the interarrival density $\tau$. First we assume that $J_1$ was distributed according to $\tau_0$ to work with a process with stationary increments. In fact if $\tau$ has finite expectation this assumption is not necessary since asymptotically the process has stationary increments (see Lindvall \cite{Lindvall}). The second assumption is that $\tau$ is described by a 1-dimensional parameter $\vartheta$. Generalising the result to a $d$-dimensional parameter should be possible at small cost, but removing all parametric assumption on $\tau$ would demand to solve a nonstandard nonparametric program for $\tau$ from the observations \eqref{eq data}: observations \eqref{eq data} only give access to truncated values of realisations of $\tau$ spaced of more than $\Delta$. Then the problem of estimating $\tau$ from \eqref{eq data} should be considered separately.

\paragraph{Other generalisations.} We constructed in the microscopic regime an adaptive minimax estimator of the jump density of a renewal reward process. The methodology presented here should adapt to any process defined similarly to $X$ but whose counting process has stationary increments and manageable dependencies. We consider in the present paper a renewal counting measure since we are interested in expanding the methodology to other regimes of $\Delta$, namely when $\Delta=\Delta_T$ tends to a constant (intermediate regime) or to infinity (macroscopic regime). The macroscopic regime is of special interest since the observed process presents diffusive or anomalous asymptotic behaviour determined by the laws $f$ and $\tau$ (see for instance Meerschaert and Scheffler \cite{Meerschaert04,Meerschaert05} or Kotulski \cite{kolt}) and many applications have a model based on a macroscopically observed renewal reward processes. For instance in physics where they are used to model particle motion (see Watkins and Credgington \cite{Watkins} or Cuppen \textit{et al.} \cite{cuppen}), in biology to model the proliferation of tumor cells (see Fedotov and Iomin \cite{Fedotov}) or lipid granule motion (see Jeon \textit{et al.} \cite{Jeon}), they are also used to model records (see Sabhapandit \cite{Sabhapandit}).

\section{Proofs}\label{Section proof}

In the sequel $ \mathfrak{C}$ denotes a constant which may vary from line to line.

\subsection{Proof of Theorem \ref{thm Renewal 1}\label{proof thm 1} \label{proof thm 1}}

\subsubsection*{Proof of part 1) of Theorem \ref{thm Renewal 1}}
To prove part 1) of Theorem \ref{thm Renewal 1} we apply the general results of Kerkyacharian and Picard \cite{KP00}. For that we establish some technical lemmas.
\begin{lemma}\label{lem loicontY boule besov}
If $f$ belongs to $\mathcal{F}(s,{\pi},\mathfrak{M})$ then for $m\geq 1$, $\mathbf{P}_\Delta[f]^{\star m}$ also belongs to $\mathcal{F}(s,{\pi},\mathfrak{M})$.
\end{lemma}
 \noindent To prove Theorem \ref{thm Renewal 1}, we use Lemma \ref{lem loicontY boule besov} for $m=1$ only, but we take some advance on the proof of Theorem \ref{thm Renewal 2}.
\begin{proof}[{Proof of Lemma \ref{lem loicontY boule besov}}]
It is straightforward to derive $\big\|\mathbf{P}_\Delta[f]^{\star m}\big\|_{L_1(\R)}=1$. The remainder of the proof is a consequence of the following result:
Let $f\in {\mathcal{B}^s_{{\pi} \infty}(\mathcal{D})}$ and $g\in L_1$ we have \begin{align}\|f\star g\|_{{\mathcal{B}^s_{{\pi} \infty}(\mathcal{D})}}\leq \|f\|_{{\mathcal{B}^s_{{\pi} \infty}(\mathcal{D})}}\|g\|_{L_1(\R)}.\label{eq etoile}\end{align}
To prove \eqref{eq etoile} we use the definition of the Besov norm \eqref{eq besov norm}; the result is a consequence of Young's inequality and elementary properties of the convolution product. First Young's inequality gives \begin{align}\label{eq ine1}\|f_1\star f_2\|_{L_\pi(\R)}\leq \|f_1\|_{L_\pi(\R)}\|f_2\|_{L_1(\R)}.\end{align} Then the differentiation property of the convolution product leads for $n\geq1$ to \begin{align}\label{eq ine2}\Big\|\frac{d^n}{dx^n}(f_1\star f_2)\Big\|_{L_\pi(\mathcal{D})}=\Big\|\Big(\frac{d^n}{dx^n} f_1\Big)\star f_2\Big\|_{L_\pi(\R)}\leq \Big\|\frac{d^n}{dx^n} f_1\Big\|_{L_\pi(\mathcal{D})}\|f_2\|_{L_1(\R)}.\end{align} Finally translation invariance of the convolution product enables to get \begin{align}\big\|\mathbf{D} ^h\mathbf{D} ^h[(f_1\star f_2)^{(n)}]\big\|_{L_\pi(\mathcal{D})}&=\big\|(\mathbf{D} ^h\mathbf{D} ^h[f_1^{(n)}])\star f_2\big\|_{L_\pi(\mathcal{D})}\nonumber\\&\leq \big\|\mathbf{D} ^h\mathbf{D} ^h[f_1^{(n)}]\big\|_{L_\pi(\mathcal{D})}\|f_2\|_{L_1(\R)}.\label{eq ine3}\end{align} Inequality \eqref{eq etoile} is then obtained by bounding $\|f\star g\|_{{\mathcal{B}^s_{{\pi} \infty}(\mathcal{D})}}$ using \eqref{eq ine1}, \eqref{eq ine2} and \eqref{eq ine3}. To complete the proof of Lemma \ref{lem loicontY boule besov}, we apply $m-1$ times \eqref{eq etoile} which leads to $$\forall m\in \N\setminus\{0\},\ \ \ \ \ \big\|\mathbf{P}_\Delta[f]^{\star m}\big\|_{{\mathcal{B}^s_{{\pi} \infty}(\mathcal{D})}}\leq \big\|\mathbf{P}_\Delta[f]\big\|_{{\mathcal{B}^s_{{\pi} \infty}(\mathcal{D})}}.$$ The triangle inequality gives $\|\mathbf{P}_\Delta[f]^{\star m}\|_{{\mathcal{B}^s_{{\pi} \infty}(\mathcal{D})}}\leq \|f\|_{{\mathcal{B}^s_{{\pi} \infty}(\mathcal{D})}}\leq\mathfrak{M}$ which concludes the proof.
\end{proof}

\begin{lemma}\label{lem Ros NC}
Let $2^{j}\leq  T$, then for $p\geq 1$ we have $$\E\big[\big|\widehat{\gamma}_{jk}-\gamma_{jk}\big|^p\big]\leq \mathfrak{C}_{p
,\|g\|_{L_p(\R)},\mathfrak{M}} T^{-p/2},$$
where $\widehat{\gamma}_{jk}$ is defined in \eqref{eq est coeff NC} and \begin{align}\label{eq coeff gam}\gamma_{jk}=\int g_{jk}(y) \mathbf{P}_\Delta[f](y)dy.\end{align}
\end{lemma}

\begin{proof}[Proof of Lemma \ref{lem Ros NC}]
The proof is obtained with Rosenthal's inequality: let $p\geq 1$ and let $(Y_1,\ldots,Y_n)$ be independent random variables such that $\E[Y_i]=0$ and $\E\big[|Y_i|^p\big]<\infty$. Then there exists $\mathfrak{C}_p$ such that \begin{align}\label{eq Ros}\E\bigg[\Big|\sum_{i=1}^nY_i \Big|^p\bigg]\leq \mathfrak{C}_p\bigg\{\sum_{i=1}^n\E\big[|Y_i|^p\big]+\Big(\sum_{i=1}^n\E\big[|Y_i|^2\big]\Big)^{p/2}\bigg\}.\end{align}

\noindent According to Proposition \ref{PropDefOperator} the $\big(\mathbf{D}^{\Delta_T} X_i\big)$ have distribution $$f_{\mathbf{D}^{\Delta_T}(x) X_1}=p(\Delta_T)\delta_0(x)+\big(1-p(\Delta_T)\big)\mathbf{P}_{\Delta_T}[f](x),\ \ \ \ x\in\mathcal{D}$$ where $\delta_0$ is the Dirac delta function and $p(\Delta_T)=\PP(R_{\Delta_T}=0)$. We derive \begin{align*}\E \big[ \widehat{\gamma}_{jk}\big]&=\int g_{jk}(z)\mathds{1}_{\{z\ne0\}}\frac{f_{\mathbf{D}^{\Delta_T} X_1}(z)}{1-p(\Delta_T)}dz = \int g_{jk}(z)\mathbf{P}_{\Delta_T}[f](z)dz =\gamma_{jk}.\end{align*} Then $ \widehat{\gamma}_{jk}-\gamma_{jk}$ is a sum of centered and identically distributed random variables, define $$Z_i=\frac{1}{1-p(\Delta_T)}g_{jk}\big(\mathbf{D}^{\Delta_T} X_i\big)\mathds{1}_{\{\mathbf{D}^{\Delta_T} X_i\ne0\}}.$$ Since $x$ is a renewal reward process, nonzero and nonconsecutive $Z_i$ are independent, then if we separate the sum in two sums of nonzero and nonconsecutive indices we can apply Rosenthal's inequality for independent variables to each sum, it wont affect the rates but the constant will modified. For $p\geq1$ we have by convex inequality
\begin{align*}
\E \big[\big|Z_i-\E[Z_i]\big|^p\big]&\leq 2^p \E \big[\big|Z_i\big|^p\big]\\&\leq\frac{2^p2^{jp/2}}{\big(1-p(\Delta_T)\big)^p}\int |g(2^jy-k)|^p\mathds{1}_{\{y\ne 0\}}f_{\mathbf{D}^{\Delta_T}}(y)dy\\
&=\frac{2^p2^{j(p/2-1)}}{\big(1-p(\Delta_T)\big)^{p-1}}\int |g(z)|^p  \mathbf{P}_{\Delta_T}[f]\Big(\frac{z+k}{2^j}\Big)dz,\end{align*} where we made the substitution $z=2^jy-k$. Lemma \ref{lem loicontY boule besov} and Sobolev embeddings (see \cite{Cohen,Donoho96,KerkPicTsyb}) \begin{align}\label{eq Sobolev embed}\mathcal{B}^s_{{\pi} \infty}\hookrightarrow \mathcal{B}^{s'}_{p \infty}\ \ \ \mbox{ and }\ \ \  \mathcal{B}^{s'}_{{\pi} \infty}\hookrightarrow\mathcal{B}^s_{{\infty} \infty},\end{align}where $p>\pi$, $s\pi>1$ and $s'=s-1/\pi+1/p$, give $\big\|\mathbf{P}_{\Delta_T}[f]\big\|_\infty\leq \mathfrak{M}$. It follows that $$\E \big[\big|Z_i-\E[Z_i]\big|^p\big]\leq 2^p2^{j(p/2-1)}\|g\|_{L_p(\R)}^p\mathfrak{M}/(1-p(\Delta_T))^{p-1}$$ and $$\E \big[\big|Z_i-\E[Z_i]\big|^2\big]\leq \mathfrak{M}/(1-p(\Delta_T))$$ since $\|g\|_{L_2(\R)}=1$. Rosenthal's inequality \eqref{eq Ros} gives for $p\geq1$ \begin{align*}
 \E\big[\big| \widehat{\gamma}_{jk}-\gamma_{jk}\big|^p\big]&\leq \mathfrak{C}_p \Big\{2^p\Big(\frac{2^{j}}{ A_{T}}\Big)^{\frac{p}{2}-1}\|g\|_{L_p(\R)}^p\mathfrak{M}+\mathfrak{M}^{p/2}\Big\}A_{T}^{-\frac{p}{2}},
 \end{align*} where $A_{T}=\lfloor T\Delta_T^{-1}\rfloor(1-p(\Delta_T))$. To conclude we use that $$1-p(\Delta_T)=\PP(J_1\geq\Delta_T)=\frac{1}{\mu}\int_0^{\Delta_T}\big(1-F(u)\big)du,$$ since $J_1$ has distribution \eqref{eq tau0}, and derive that there exists $\Delta_1>0$ such that $ F(\Delta_1)\leq \frac{1}{2}$ and for all $\Delta_T\leq \Delta_1$ we have \begin{align}\label{eq pcontrol Renew} \frac{\Delta_T}{2\mu}\leq 1-p(\Delta_T)\leq \frac{\Delta_T}{\mu}.\end{align} It follows that $$\frac{T}{2\mu}\leq A_T \leq \frac{T}{\mu}$$ and then using $2^j\leq T$ \begin{align*}
 \E\big[\big| \widehat{\gamma}_{jk}-\gamma_{jk}\big|^p\big]&\leq\mathfrak{C}_{p,\|g\|_{L_p(\R)},\mathfrak{M},\mu}T^{-p/2}.
 \end{align*} The proof the complete.\end{proof}

\begin{lemma}\label{lem Bern NC} Choose $j$ and $c$ such that $$2^j T^{-1}\log(T^{1/2})\leq1\mbox{ and }c^2\geq 32\mu\Big(\mathfrak{M}+\frac{c\|g\|_\infty}{6}\Big).$$ For all $r\geq 1$, let $\kappa_r=cr$. We have $$\PP\Big(\big|\widehat{\gamma}_{jk}-\gamma_{jk}\big|\geq \frac{\kappa_r}{2}T^{-1/2}\sqrt{\log(T^{1/2})}\Big)\leq T^{-r/2},$$
where $\widehat{\gamma}_{jk}$ is defined in \eqref{eq est coeff NC} and $\gamma_{jk}$ in \eqref{eq coeff gam}.
\end{lemma}

\begin{proof}[Proof of Lemma \ref{lem Bern NC}]The proof is obtained with Bernstein's inequality. Consider $Y_1,\ldots,Y_n$ independent random variables such that $|Y_i|\leq \mathfrak{A}$, $\E[Y_i]=0$ and $b_n^2=\sum_{i=1}^n\E[Y_i^2]$. Then for any $\lambda>0$, \begin{align}\label{eq Bernstein}\PP\Big(\Big|\sum_{i=1}^nY_i\Big| >\lambda\Big)\leq 2\exp\Big(-\frac{\lambda^2}{2(b_n^2+\frac{\lambda \mathfrak{A}}{3})}\Big).\end{align} We keep notation $Z_i$ introduced in the proof of Lemma \ref{lem Ros NC}, $ \widehat{\gamma}_{jk}-\gamma_{jk}$ is a sum of centered and identically distributed random variables bounded by $2^{j/2}\|g\|_\infty/(1-p(\Delta_T))$ which verify $$\E\big[ \big|Z_i-\E[Z_i]\big|^2\big]\leq\mathfrak{ M}/(1-p(\Delta_T)).$$ After separating the sum to get two sums of nonzero and nonconsecutive indices we apply Bernstein's inequality \eqref{eq Bernstein} for independent variables to each sum, which modify the constants. It follows that
 \begin{align*}
 \PP\Big(|\widehat{\gamma}_{jk}-\gamma_{jk}|&\geq \frac{\kappa_r}{2}T^{-1/2}\sqrt{\log(T^{1/2})}\Big)\\
 \leq&2\exp\Bigg(-\frac{\kappa_r^2T^{-1}\log(T^{1/2}) \lfloor T\Delta_T^{-1}\rfloor\big(1-p(\Delta_T)\big)} {16\Big( \mathfrak{M} +\frac{\kappa_r  T^{-1/2}\sqrt{\log(T^{1/2})}2^{j/2}\|g\|_\infty}{6}\Big)}\Bigg). \end{align*} Using that $2^jT^{-1}\log(T^{1/2})\leq1$ and \eqref{eq pcontrol Renew} which gives $$T^{-1}\lfloor T\Delta_T^{-1}\rfloor(1-p(\Delta_T)) \geq \frac{1}{2\mu},$$ we have \begin{align*}
 \PP\Big(|\widehat{\gamma}_{jk}-\gamma_{jk}|&\geq \frac{\kappa_r}{2}T^{-1/2}\sqrt{\log(T^{1/2})}\Big)\\
 \leq&2\exp\bigg(-\frac{c^2r} {32\mu\big(\mathfrak{M}+\frac{\kappa_r\|g\|_\infty}{6}\big)}r\log(T^{1/2})\bigg)\leq T^{-r/2},
 \end{align*} since $c^2\geq 32\mu\big(\mathfrak{M}+\frac{c\left\|g\right\|_\infty}{6}\big).$ The proof is complete.\end{proof}

\begin{proof}[Proof of of part 1) of Theorem \ref{thm Renewal 1}]
It is a consequence of Lemma \ref{lem loicontY boule besov}, \ref{lem Ros NC}, \ref{lem Bern NC} and of the general theory of wavelet threshold estimators of Kerkyacharian and Picard \cite{KP00}. It suffices to have conditions (5.1) and (5.2) of Theorem 5.1 of \cite{KP00}, which are satisfied --Lemma \ref{lem Ros NC} and \ref{lem Bern NC}-- with $c(T)=T^{-1/2}$ and $\Lambda_n=c(T)^{-1}$ (with the notation of \cite{KP00}). We can now apply Theorem 5.1, its Corollary 5.1 and Theorem 6.1 of \cite{KP00} to obtain the result.
\end{proof}

\subsubsection*{Completion of the proof of Theorem \ref{thm Renewal 1}}
To prove part 2) of Theorem \ref{thm Renewal 1} we decompose the $L_p$ loss as follows \begin{align*}\big(\E \big[ \| \widehat{f}_{T,\Delta_T}-&f\|_{L_p(\mathcal{D})}^p\big]\big)^{1/p}\\&\leq\big(\E\big[\big\|\widehat{f}_{T,\Delta_T}-\mathbf{P}_{\Delta_T}[f] \|_{L_p(\mathcal{D})}^p\big]\big)^{1/p} +\big\|\mathbf{P}_{\Delta_T}[f]  -f\big\|_{L_p(\mathcal{D})}.\end{align*}
An upper bound for the first term is given by part 1) of Theorem \ref{thm Renewal 1} \begin{align}\label{eq prf thm11}\big(\E\big[\big\|\widehat{f}_{T,\Delta_T}-\mathbf{P}_{\Delta_T}[f] \|_{L_p(\mathcal{D})}^p\big]\big)^{1/p} &\leq \mathfrak{C} T^{-\alpha(s,p,\pi)},\end{align} where $\mathfrak{C}$ continuously depends on $\mu$, and on $s,\pi,p,\mathfrak{M},\phi,\psi$ and $\mu$. Since \begin{align*}
\mathbf{P}_{\Delta_T}[f]-f=-(1-p_1(\Delta_T))f+\sum_{m=2}^{\infty}p_m(\Delta_T)\mathbf{P}_{\Delta_T}[f]^{\star m}
\end{align*}Lemma \ref{lem pmControl}, Young's inequality, which gives $\|\mathbf{P}_{\Delta_T}[f]^{\star m}\|_{L_p(\mathcal{D})}\leq \|f\|_{L_p(\mathcal{D})}$ and Sobolev embeddings \eqref{eq Sobolev embed}, which give $\|f\|_{L_p(\mathcal{D})}\leq \mathfrak{M}$, enable to get the bound \begin{align}\label{eq prf thm12}\big\|\mathbf{P}_{\Delta_T}[f]  -f\big\|_{L_p(\mathcal{D})}&\leq 2\tau(0)\Delta_T+2\mathfrak{M}\sum_{m=2}^\infty \frac{\big(2\tau(0)\Delta_T\big)^{m-1}}{m!}\leq \mathfrak{C}\Delta_T,\end{align} where $\mathfrak{C}$ continuously depends on $\tau(0)$ and $\mathfrak{M}$. We finish the proof noticing that \eqref{eq prf thm11} is predominant in front of \eqref{eq prf thm12} since $\alpha(s,p,\pi)\leq 1/2$ and $T\Delta_T^2=O(1)$. Finally we take the supremum in $\mu$ and $\tau(0)$ over any compact of $(0,\infty)$ to render the constant independent of the unknown interarrival law $\tau$. The proof is now complete.

\subsection{Proof of Proposition \ref{prop complet contract}}

First we prove part 1) of Proposition \ref{prop complet contract}. The set $\mathcal{H}(s,{\pi},\mathfrak{O},\mathfrak{N})$ is a subset of $$\big(\mathcal{B}^s_{{\pi} \infty}(\mathcal{D}),\|.\|_{\mathcal{B}^s_{{\pi} \infty}(\mathcal{D})}\big)$$ which is a Banach space. We show that $\mathcal{H}(s,{\pi},\mathfrak{O},\mathfrak{N})$ is complete since it is a closed subset of a Banach space. For that we establish the following assertions; for all sequence $h_n\in \mathcal{H}(s,{\pi},\mathfrak{O},\mathfrak{N})$ such that there exists $h$ with $$\|h_n-h\|_{\mathcal{B}^s_{{\pi} \infty}(\mathcal{D})}\rightarrow 0, \ \ \ \mbox{ as } n\rightarrow \infty,$$ we have $h\in \mathcal{H}(s,{\pi},\mathfrak{O},\mathfrak{N})$ \textit{i.e.} $\|h\|_{\mathcal{B}^s_{{\pi} \infty}(\mathcal{D})}\leq\mathfrak{N}$ and $\|h\|_{L_1(\mathcal{D})}\leq\mathfrak{O}$. The first inequality is immediate. The second one is a consequence of the compactness of $\mathcal{D}$. Indeed $$\|h_n-h\|_{\mathcal{B}^s_{{\pi} \infty}(\mathcal{D})}\rightarrow 0, \ \ \ \mbox{ as } n\rightarrow \infty,$$ we have by definition of the Besov norm \eqref{eq besov norm} that $$\|h_n-h\|_{L_\pi(\mathcal{D})}\rightarrow 0, \ \ \ \mbox{ as } n\rightarrow \infty.$$ Since $\mathcal{D}$ is compact and $\pi\geq1$ we derive from H\"older's inequality that $$\|h_n-h\|_{L_1(\mathcal{D})}\rightarrow 0, \ \ \ \mbox{ as } n\rightarrow \infty$$ and then $\|h\|_{L_1(\mathcal{D})}\leq \mathfrak{O}$ follows. The proof of part 1) of Proposition \ref{prop complet contract} is now complete.

\

To prove part 2) of Proposition \ref{prop complet contract} we show that $\mathbf{H}_{\Delta,f}$ sends elements of $\mathcal{H}(s,{\pi},\mathfrak{O},\mathfrak{N})$ into $\mathcal{H}(s,{\pi},\mathfrak{O},\mathfrak{N})$ and that it is a contraction. We start with the first assertion,
the triangular inequality gives for $h\in\mathcal{H}(s,{\pi},\mathfrak{O},\mathfrak{N})$ \begin{align*}\big\|\mathbf{H}_{\Delta,f}[h]\big\|_{L_1(\mathcal{D})}&\leq\big\|\mathbf{P}_\Delta[f]\big\|_{L_1(\mathcal{D})} +\big(1-p_1(\Delta)\big)\|h\|_1+\sum_{m=2}^\infty p_m(\Delta)\|h^{\star m}\|_{L_1(\mathcal{D})},\end{align*} where $\big\|\mathbf{P}_\Delta[f]\big\|_{L_1(\mathcal{D})}\leq \big\|\mathbf{P}_\Delta[f]\big\|_{L_1(\R)} =1$. Immediate induction on Young's inequality leads to $$\|h^{\star m}\|_{L_1(\mathcal{D})}\leq \|h\|_{L_1(\mathcal{D})}^m\leq \mathfrak{O}^m$$ since $h\in \mathcal{H}(s,{\pi},\mathfrak{O},\mathfrak{N})$ and with Lemma \ref{lem pmControl} we get \begin{align*}\big\|\mathbf{H}_{\Delta,f}[h]\big\|_1&\leq 1+2\mathfrak{O}\tau(0)\Delta +\frac{1}{\tau(0)\Delta}(e^{2\mathfrak{O}\tau(0)\Delta}-1-2\mathfrak{O}\tau(0)\Delta)\leq\mathfrak{O}\end{align*} for $\Delta$ small enough since $\mathfrak{O}>1$. Similar computations and \eqref{eq etoile} give \begin{align*}
 \big\|\mathbf{H}_{\Delta,f}[h]\big\|_{\mathcal{B}^s_{{\pi} \infty}(\mathcal{D})}\leq&\big\|\mathbf{P}_\Delta[f]\big\|_{\mathcal{B}^s_{{\pi} \infty}(\mathcal{D})}+\big(1-p_1(\Delta)\big)\|h\|_{\mathcal{B}^s_{{\pi} \infty}(\mathcal{D})}\\&+\sum_{m=2}^\infty p_m(\Delta)\|h^{\star m}\|_{\mathcal{B}^s_{{\pi} \infty}(\mathcal{D})}\\
 \leq& \mathfrak{M}+2\tau(0)\Delta\mathfrak{N}+\frac{\mathfrak{N}}{\tau(0)\Delta\mathfrak{O}} (e^{2\mathfrak{O}\tau(0)\Delta}-1-2\mathfrak{O}\tau(0)\Delta) \leq\mathfrak{N},
 \end{align*} for $\Delta$ small enough since $\mathfrak{M}<\mathfrak{N}$. Then if $h$ is in $\mathcal{H}(s,{\pi},\mathfrak{O},\mathfrak{N})$, $\mathbf{H}_{\Delta,f}$ belongs to $\mathcal{H}(s,{\pi},\mathfrak{O},\mathfrak{N})$.

 For the contraction property, we have for all $h_1,h_2 \in \mathcal{H}(s,{\pi},\mathfrak{O},\mathfrak{N})$\begin{align}
\mathbf{H}_{\Delta,f}[h_1]-\mathbf{H}_{\Delta,f}[h_2]=&(1-p_1(\Delta))(h_1-h_2)\nonumber\\&- (h_1-h_2)\star \sum_{m=2}^{\infty}p_m(\Delta)\sum_{q=0}^{m-1}h_1^{\star q}\star h_2^{\star m-1-q}\label{eq prof cont1}
\end{align} Lemma \ref{lem pmControl} gives \begin{align}\label{eq prof cont2}0\leq1-p_1(\Delta)\leq 2\tau(0)\Delta,\end{align} and with Young's inequality and since $h_1$ and $h_2$ belong to $\mathcal{H}(s,{\pi},\mathfrak{O},\mathfrak{N})$ we get
\begin{align}\label{eq prof cont3}
\Big\|\sum_{m=2}^{\infty}p_m(\Delta)\sum_{q=0}^{m-1}h_1^{\star q}\star h_2^{\star m-1-q}\Big\|_{L_1(\mathcal{D})}&\leq 2\mathfrak{O}\big(e^{2\tau(0)\mathfrak{O}\Delta}-1\big)\end{align} for $\Delta$ small enough. Finally injecting \eqref{eq prof cont2} and \eqref{eq prof cont3} into \eqref{eq prof cont1} leads to the contraction property for all $h_1,h_2 \in \mathcal{H}(s,{\pi},\mathfrak{O},\mathfrak{N})$ $$\big\|\mathbf{H}_{\Delta,f}[h_1]-\mathbf{H}_{\Delta,f}[h_2]\big\|_{\mathcal{B}^s_{{\pi} \infty}(\mathcal{D})} \leq \big(2\tau(0)\Delta+2\mathfrak{O}\big(e^{2\tau(0)\mathfrak{O}\Delta}-1\big)\big)\|h_1-h_2\|_{\mathcal{B}^s_{{\pi} \infty}(\mathcal{D})},$$ which concludes the proof.

\subsection{Proof of Theorem \ref{thm Renewal 2}\label{section proof thm2}}

\subsubsection*{Preliminary}

The estimators of the convolution powers of $\mathbf{P}_\Delta[f]$ depend on $N_T$ which is random and depends on the $\mathbf{D}^\Delta_m X_i$.

 \begin{lemma}\label{lem Nconcentre Renewal}
 Work under Assumption \ref{ass queue tau} and let $p(\Delta)=\PP(R_{\Delta}\ne0)$ and $\Delta_1$ be such that $\int_0^{\Delta_1}\tau(x)dx\leq \frac{1}{2}.$ Then for all $\lambda>0$ and $\Delta\leq\Delta_1$ we have
 \begin{align*}
 \PP\Big(\Big|\frac{N_T}{\lfloor T\Delta^{-1}\rfloor}-p(\Delta)\Big|>\lambda \Delta\Big)\leq \exp\Big(-\mathfrak{C}\sqrt{T}\Delta\Big),
 \end{align*} where $\mathfrak{C}$ depends on $\mathfrak{A},\mathfrak{a},\mathfrak{g},\mu,\lambda$.
 \end{lemma}

 \begin{proof}[Proof of Lemma \ref{lem Nconcentre Renewal}]  We have $$\frac{N_T}{\lfloor T\Delta^{-1}\rfloor}-p(\Delta)=\frac{1}{\lfloor T\Delta^{-1}\rfloor}\sum_{i=1} ^{\lfloor T\Delta^{-1}\rfloor}Y_i,$$ where $$Y_i=\mathds{1}_{\{\mathbf{D}^{\Delta}X_i\ne0\}}-p(\Delta),\ \ \ \ i=1,\ldots,\lfloor T\Delta^{-1}\rfloor$$ are centered random variables, bounded by $M=1-p(\Delta)$ and such that $$\E[Y_i^2]\leq p(\Delta).$$ To show the result, we apply Theorem 4.5 of Dedecker \textit{et. al.} \cite{Doukhan} which is a Bernstein-type inequality for dependent data. We have to verify conditions (4.4.16) and (4.4.17) of Theorem 4.5 of \cite{Doukhan}. With their notation, condition (4.4.16) ensures that for all $u$-tuples $(s_1,\ldots,s_u)$ and all $v$-tuples $(t_1,\ldots,t_v)$ such that $$1\leq s_1\leq\ldots\leq s_u\leq t_1\leq\ldots \leq t_v\leq\lfloor T\Delta^{-1}\rfloor$$ we have
 $$\big|Cov(Y_{s_1}\ldots Y_{s_u},Y_{t_1}\ldots Y_{t_v})\big|\leq K^2M^{u+v-2}uv \rho(t_1-s_u),$$ for some positive constant $K$ and a nonincreasing function $\rho$ satisfying (4.4.17) namely $$\sum_{s=0}^\infty (s+1)^k\rho(s)\leq L_1L_2^k(k!)^\nu,\ \ \ \forall k\geq0,$$ where $L_1$, $L_2$ and $\nu$ are positive constants.

 Since $X$ is a renewal process, $Y_{s_1}\ldots Y_{s_u}$ and $Y_{t_1}\ldots Y_{t_v}$ are independent if there exists $r$ such that $s_u< r <t_1$ and $Y_r=1-p(\Delta)$ \textit{i.e} there is a jump between $Y_{s_u}$ and $Y_{t_1}$. For the covariance to be nonzero it is necessary that no jump occurred between $s_u\Delta$ and $(t_1-1)\Delta$. Let $s=t_1-s_u-1$ using that $R$ is stationary we get an upper bound for $\rho$ \begin{align}\label{eq upp rho}\rho(t_1-s_u)\leq\PP\big(R_{(t_1-1)\Delta}-R_{s_u\Delta}=0\big)=\PP\big(R_{s\Delta}=0\big)=\int_{s\Delta}^\infty \tau_0(x)dx\end{align} which decreases with $s$. Moreover since the $Y_i$ are centered and bounded by $M\leq 1$ we have by Cauchy-Schwarz and $\E[Y_i^2]\leq p(\Delta)$ \begin{align*}\big|Cov(Y_{s_1}\ldots Y_{s_u},Y_{t_1}\ldots Y_{t_v})\big|&\leq \big|Cov(Y_{s_1},Y_{t_v})\big|\leq \sqrt{\E\big[Y_{s_1}^2\big]\E\big[Y_{t_v}^2\big]}\leq p(\Delta).\end{align*} We deduce that condition (4.4.16) is fulfilled with $K=p(\Delta)^{1/2}$ and the nonincreasing sequence $\rho$.

Next we show that $\rho$ satisfies (4.4.17), using Assumption \ref{ass queue tau} and \eqref{eq upp rho} we get for $s\geq1$ \begin{align*}
\rho(s)&\leq \frac{1}{\mu}\int_{s\Delta_T}^\infty \big(1-\int_0^x \tau(t)dt\big)dx\leq \mathfrak{C}\exp\big(-\mathfrak{a}(s\Delta)^\mathfrak{g'}\big),
\end{align*} where $\mathfrak{g}<\mathfrak{g'}$ and $\mathfrak{C}$ depends on $\mathfrak{A},\mathfrak{a},\mu,\mathfrak{g}$.
Which leads to for $k\geq0$ \begin{align}\sum_{s=1}^\infty s^k\rho(s)&\leq \sum_{s=1}^{1/\Delta}s^k+\mathfrak{C}\sum_{s=1/\Delta}^\infty s^k\exp\big(-\mathfrak{a}(s\Delta)^\mathfrak{g'}\big)\nonumber\\&\leq \Delta^{-(k+1)}+\mathfrak{C}\Delta^{-k}\sum_{s'=1}^\infty s'^{k}\exp\big(-\mathfrak{a}(s')^\mathfrak{g'}\big)\nonumber\\&\leq \mathfrak{C}\Delta^{-(k+1)}\label{eq prf bernDep}\end{align} where $\mathfrak{C}$ depends on $\mathfrak{A},\mathfrak{a},\mathfrak{g},\mu$, condition (4.4.17) follows with $L_1=\mathfrak{C}\Delta^{-1}$, $L_2=\Delta^{-1}$ and $\nu=0$.

We can now apply Theorem 4.5 which gives for all $\lambda>0$
\begin{align*}
\PP\Big(\Big|&\frac{N_T}{\lfloor T\Delta^{-1}\rfloor}-p(\Delta)\Big|>\lambda \Delta\Big)\\&\leq 2\exp\bigg(-\frac{\lfloor T\Delta^{-1}\rfloor^2\Delta^2\lambda^2}{2\big(\lfloor T\Delta^{-1}\rfloor p(\Delta)+(\lfloor T\Delta^{-1}\rfloor\Delta\lambda)^{3/2}\sqrt{2^5\Delta^{-2}}\big)}\bigg).
\end{align*} Using \eqref{eq pcontrol Renew} we derive for $\Delta\leq \Delta_1$ \begin{align*}\PP\Big(\Big|\frac{N_T}{\lfloor T\Delta^{-1}\rfloor}-p(\Delta)\Big|>\lambda \Delta\Big)&\leq \exp\Big(-\mathfrak{C}\sqrt{T}\Delta\Big),
\end{align*} where $\mathfrak{C}$ depends on $\lambda,\mathfrak{A},\mathfrak{a},\mathfrak{g},\mu$. The proof is now complete.
 \end{proof}

\subsubsection*{Proof of part 1) of Theorem \ref{thm Renewal 2}}
As for the proof of part 1) of Theorem \ref{thm Renewal 1} we apply the general results of Kerkyacharian and Picard \cite{KP00} and first establish some technical lemmas.

\begin{lemma}\label{lem Ros C}
Let $2^{j}\leq  T$, then for $p\geq 1$ we have for all $m\geq1$ $$\E\big[\big|\widehat{\gamma}_{jk}^{(m)}-\gamma_{jk}^{(m)}\big|^p\big]\leq \mathfrak{C}_{p,m,\|g\|_{L_p(\R)},\mathfrak{M},\mu,\tau} T^{-p/2},$$
where $\widehat{\gamma}_{jk}^{(m)}$ is defined in \eqref{eq est Coeffconvol} and \begin{align}\label{eq coeff gamma}\gamma_{jk}^{(m)}=\int g_{jk}(y) \mathbf{P}_\Delta[f]^{\star m}(y)dy.\end{align}
\end{lemma}

\begin{proof}[{Proof of Lemma \ref{lem Ros C}}] For $m\geq 1$, $\widehat{\gamma}_{jk}^{(m)}-\gamma_{jk}^{(m)}$ is the sum of $\lfloor N_T/m\rfloor$ identically distributed random variables, where $N_T$ is random. First we replace $N_T$ by its deterministic asymptotic limit using the following decomposition \begin{align*}\E\big[\big|\widehat{\gamma}_{jk}^{(m)}-\gamma_{jk}^{(m)}\big|^p\big]= & \E\big[\big|\widehat{\gamma}_{jk}^{(m)}-\gamma_{jk}^{(m)}\big|^p\mathds{1}_{\big\{\big|\frac{N_T}{\lfloor T\Delta_T^{-1}\rfloor}-p(\Delta_T)\big|\geq\lambda\big\}}\big]\\&+\E\big[\big|\widehat{\gamma}_{jk}^{(m)}-\gamma_{jk}^{(m)}\big|^p \mathds{1}_{\big\{|\frac{N_T}{\lfloor T\Delta_T^{-1}\rfloor}-p(\Delta_T)\big|<\lambda\big\}}\big].\end{align*}  Take $\lambda=1/4\mu$ and denote $n_m=\big\lfloor T/m\mu\big\rfloor$ and $n'_m=\big\lfloor T/(4m\mu)\big\rfloor$, we have with that \eqref{eq pcontrol Renew} that\begin{align*} \E\big[\big|\widehat{\gamma}_{jk}^{(m)}-\gamma_{jk}^{(m)}\big|^p\big]
\leq& 2^{jp/2}\|g\|_\infty \PP\Big(\Big|\frac{N_T}{\lfloor T\Delta_T^{-1}\rfloor}-p(\Delta_T)\Big|\geq\frac{\Delta_T}{4\mu}\Big)\\
&+\E\Big[\Big|\frac{1}{n'_m}\sum_{i=1}^{n_m}\big(\mathbf{D}^{\Delta}_m X_{S_i}-\E\big[\mathbf{D}^{\Delta}_m X_{S_i}\big] \big)\Big|^p\Big].\end{align*} For the first term of the right hand part, $T\Delta_T=O(T^\delta)$, Lemma \ref{lem Nconcentre Renewal} and $2^{j}\leq T$ leads to \begin{align}\PP\Big(\big|\frac{N_T}{\lfloor T\Delta_T^{-1}\rfloor}-p(\Delta_T)\big|\geq\frac{\Delta_T}{4\mu}\Big)& \leq \mathfrak{C}T^{p/2}\exp\Big(-\mathfrak{C}T^\delta\Big)\leq \mathfrak{C}\exp(-T^{\delta'_p})\label{eq prf Ros C1},\end{align} for some $\delta'_p<\delta$ and where $\mathfrak{C}$ depends on $p,\|g\|_\infty,\mathfrak{A},\mathfrak{a},\mu$. For the second term we apply Rosenthal's inequality \eqref{eq Ros}. Since $X$ is a renewal process the variables $\big(\mathbf{D}^{\Delta}_m X_{S_{2i}}\big)_i$ are independent but dependent of the variables $\big(\mathbf{D}^{\Delta}_m X_{S_{2i+1}}\big)_i$ which are independent. It ensures that the variables $\big(\mathbf{D}^{\Delta}_m X_{S_i}\big)$ are distributed according to $\mathbf{P}_{\Delta_T}[f]^{\star m}$. Moreover if we separate the sum $\widehat{\gamma}_{jk}^{(m)}-\gamma_{jk}^{(m)}$ between odd and even indices we can apply Rosenthal's inequality for independent variables to each sum. For $p\geq1$ we have by convex inequality
\begin{align*}
\E \big[\big|g_{jk}(\mathbf{D}^{\Delta}_m X_{S_i})-\gamma_{jk}\big|^p\big]&\leq 2^p \E \big[\big|g_{jk}(\mathbf{D}^{\Delta}_m X_{S_i})\big|^p\big]\\&\leq 2^p2^{jp/2}\int |g(2^jy-k)|^p\mathbf{P}_{\Delta_T}[f]^{\star m}(y)dy \\&\leq 2^p2^{j(p/2-1)}\int |g(z)|^p\mathbf{P}_{\Delta_T}[f]^{\star m}\Big(\frac{z+k}{2^j}\Big)dz,\end{align*} where we made the substitution $z=2^jy-k$. Lemma \ref{lem loicontY boule besov} and Sobolev embeddings \eqref{eq Sobolev embed} give $\big\|\mathbf{P}_{\Delta_T}[f]^{\star m}\big\|_\infty\leq \mathfrak{M}$. It follows that $$\E \big[\big|g_{jk}(\mathbf{D}^{\Delta}_m X_{S_i})-\gamma_{jk}\big|^p\big]\leq 2^p2^{j(p/2-1)}\|g\|_{L_p(\R)}^p\mathfrak{M}$$ and $$\E \big[\big|g_{jk}(\mathbf{D}^{\Delta}_m X_{S_i})-\gamma_{jk}\big|^2\big]\leq \mathfrak{M}$$ since $\|g\|_{L_2(\R)}=1$. We derive for $p\geq1$ \begin{align}
\E\Big[\Big|\frac{1}{n'_m}\sum_{i=1}^{n_m}&\big(\mathbf{D}^{\Delta}_m X_{S_i}-\E\big[\mathbf{D}^{\Delta}_m X_{S_i}\big] \big)\Big|^p\Big]\nonumber\\ &\leq \mathfrak{C}_p \Big\{2^p\Big(\frac{2^{j}}{ n_m}\Big)^{\frac{p}{2}-1}\|g\|_{L_p(\R)}^p\mathfrak{M}+\mathfrak{M}^{p/2}\Big\}{n_m'}^{-p/2}\nonumber\\
&\leq \mathfrak{C}_{p,m,\|g\|_{L_p(\R)},\mathfrak{M},\mu}T^{-p/2}.\label{eq prf Ros C2}
\end{align} It follows from \eqref{eq prf Ros C1} and \eqref{eq prf Ros C2} that\begin{align*}\E\big[\big|\widehat{\gamma}_{jk}^{(m)}-\gamma_{jk}^{(m)}\big|^p\big]\leq & \mathfrak{C}\exp(-T^{\delta'_p})+\mathfrak{C}T^{-\frac{p}{2}}\leq \mathfrak{C}T^{-\frac{p}{2}},\end{align*} since the first term is negligible in front of the second as $\delta'>0$ where $\mathfrak{C}$ depends on $p,m,\|g\|_{L_p(\R)},\|g\|_\infty,\mathfrak{A},\mathfrak{a},\mathfrak{M},\mu$. It concludes the proof.
\end{proof}

\begin{lemma}\label{lem Bern C} Choose $j$ and $c$ such that $$2^jT^{-1}\log(T^{1/2})\leq1\mbox{ and }c^2\geq 256m\mu\Big(\mathfrak{M}+\frac{c\|g\|_\infty}{24}\Big).$$ For all $r\geq 1$ let $\kappa_r=cr$. We have for all $m\geq1$ $$\PP\Big(|\widehat{\gamma}_{jk}^{(m)}-\gamma_{jk}^{(m)}|\geq \frac{\kappa_r}{2}T^{-1/2}\sqrt{\log(T^{1/2})}\Big)\leq T^{-r/2},$$
where $\widehat{\gamma}_{jk}^{(m)}$ is defined in \eqref{eq est Coeffconvol} and $\gamma_{jk}^{(m)}$ in \eqref{eq coeff gamma}.
\end{lemma}

\begin{proof}[{Proof of Lemma \ref{lem Bern C}}]
As for the proof of Lemma \ref{lem Ros C} we decompose as follow for $m\geq1$ \begin{align*}\PP\Big(|\widehat{\gamma}_{jk}^{(m)}&-\gamma_{jk}^{(m)}|\geq \frac{\kappa_r}{2}T^{-1/2}\sqrt{\log(T^{1/2})}\Big)\\&\leq \PP\Big(\Big|\frac{N_T}{\lfloor T\Delta_T^{-1}\rfloor}-p(\Delta_T)\Big|\geq\frac{\Delta_T}{4\mu}\Big)\\ &+ \PP\Big(\big|\frac{1}{n'_m}\sum_{i=1}^{n_m}\big(\mathbf{D}^{\Delta}_m X_{S_i}-\E\big[\mathbf{D}^{\Delta}_m X_{S_i}\big] \big)\big|\geq \frac{\kappa_r}{2}T^{-1/2}\sqrt{\log(T^{1/2})}\Big),
\end{align*} where $n_m=\big\lfloor T/m\mu\big\rfloor$ and $n'_m=\big\lfloor T/(4m\mu)\big\rfloor$.
From $T\Delta_T^2=O(T^\delta)$ and Lemma \ref{lem Nconcentre Renewal} we derive \begin{align}\label{eq prf Bern C1}\PP\Big(\Big|\frac{N_T}{\lfloor T\Delta_T^{-1}\rfloor}-p(\Delta_T)\Big|\geq\frac{\Delta_T}{4\mu}\Big)& \leq \exp\Big(-\mathfrak{C}T^\delta\Big),\end{align} where $\mathfrak{C}$ depends on $\mathfrak{A},\mathfrak{a},\mu$. For the second term we apply Bernstein's inequality \eqref{eq Bernstein} and as in the proof of Lemma \ref{lem Ros C} we separate the sum between odd and even indices to work with independent variables. We get \begin{align*}
 \PP\Big(\Big|\frac{1}{n'_m}\sum_{i=1}^{n_m}&\big(\mathbf{D}^{\Delta}_m X_{S_i}-\E\big[\mathbf{D}^{\Delta}_m X_{S_i}\big] \big)\Big|\geq \frac{\kappa_r}{2}T^{-1/2}\sqrt{\log(T^{1/2})}\Big)\\
 \leq&2\exp\Bigg(-\frac{\kappa_r^2{n'_m}^2T^{-1}\log(T^{1/2})} {16\Big(n_m\mathfrak{M}+\frac{\kappa_rn'_mT^{-1/2}\sqrt{\log(T^{1/2})}2^{j/2}\|g\|_\infty}{6}\Big)}\Bigg)\\
 \leq&2\exp\Bigg(-\frac{c^2r} {128m\mu\Big(\mathfrak{M}+\frac{\kappa_rT^{-1/2}\sqrt{\log(T^{1/2})}2^{j/2}\|g\|_\infty}{24}\Big)}r\log(T^{1/2})\Bigg).
 \end{align*} With $2^jT^{-1}\log(T^{1/2})\leq1$ and $c^2\geq 256m\mu\big(\mathfrak{M}+\frac{c\|g\|_\infty}{24}\big)$ we have for $r\geq1$ \begin{align}\label{eq prf Bern C2}
 \PP\Big(\Big|\frac{1}{n'_m}\sum_{i=1}^{n_m}&\big(\mathbf{D}^{\Delta}_m X_{S_i}-\E\big[\mathbf{D}^{\Delta}_m X_{S_i}\big] \big)\Big|\geq \frac{\kappa_r}{2}T^{-1/2}\sqrt{\log(T^{1/2})}\Big)
 \leq v(T)^r. \end{align}It follows from \eqref{eq prf Bern C1} and \eqref{eq prf Bern C2} that \begin{align*}\PP\Big(|\widehat{\gamma}_{jk}^{(m)}&-\gamma_{jk}^{(m)}|\geq \frac{\kappa_r}{2}T^{-1/2}\sqrt{\log(T^{1/2})}\Big)&\leq \exp\Big(-\mathfrak{C}T^\delta\Big) +T^{-r/2}\leq T^{-r/2}
\end{align*} since the first term is negligible in front of the second since $\delta>0$. It concludes the proof.
\end{proof}

\begin{proof}[Completion of the proof of part 1) of Theorem \ref{thm Renewal 2}]It is a consequence of Lemma \ref{lem loicontY boule besov}, \ref{lem Ros C}, \ref{lem Bern C} and of the general theory of wavelet threshold estimators of Kerkyacharian and Picard \cite{KP00}. It suffices to have conditions (5.1) and (5.2) of Theorem 5.1 of \cite{KP00}, which are satisfied --Lemma \ref{lem Ros C} and \ref{lem Bern C}-- with $c(T)=T^{-1/2}$ and $\Lambda_n=c(T)^{-1}$ (with the notation of \cite{KP00}). We can now apply Theorem 5.1, its Corollary 5.1 and Theorem 6.1 of \cite{KP00} to obtain the result.
\end{proof}
\subsubsection*{Completion of the proof of Theorem \ref{thm Renewal 2}}
To prove Theorem \ref{thm Renewal 2} we define for $K$ in $\N$ and $x$ in $\mathcal{D}$ the quantity
 \begin{align*}
\widetilde{f}^K_{T,\Delta}(x)&=\sum_{m=1}^{K+1} l_m(\Delta,\vartheta) \widehat{P_{\Delta, m}}(x).
 \end{align*} It is the estimator of $f$ one would compute if $\tau$ were known. We decompose the $L_p$ error as follows \begin{align*}
\big(\E\big[\|\widehat{f}^K_{T,\Delta_T}-f\|_{L_p(\mathcal{D})}^p\big]\big)^{1/p}\leq& \big(\E\big [\|\widehat{f}^K_{T,\Delta_T}-\widetilde{f}^K_{T,\Delta_T}\|_{L_p(\mathcal{D})}^p\big]\big)^{1/p}\\ &+\big(\E\big[\|\widetilde{f}^K_{T,\Delta_T}-f\|_{L_p(\mathcal{D})}^p\big]\big)^{1/p},
\end{align*} and control each term separately.

First we look at the second term \begin{align}\big(\E \big[ \| \widetilde{f}^K_{T,\Delta_T} -f\|_{L_p(\mathcal{D})}^p\big]\big)^{1/p}\leq&\big(\E\big[\big\|\widetilde{f}^K_{T,\Delta_T}-\mathbf{L}_{\Delta_T,K} \|_{L_p(\mathcal{D})}^p\big]\big)^{1/p} \nonumber\\&+\big\|\mathbf{L}_{\Delta_T,K} -\mathbf{H}_{\Delta,f}^{\circ K}\big[\mathbf{P}_{\Delta_T}[f]\big\|_{L_p(\mathcal{D})}\nonumber\\&+\big\|\mathbf{H}_{\Delta,f}^{ \circ K}\big[\mathbf{P}_{\Delta_T}[f]\big] -f\big\|_{L_p(\mathcal{D})}.\label{eq prf thm2 1}\end{align}
An upper bound for the first term is given by part 1) of Theorem \ref{thm Renewal 2}, given the definition \eqref{eq LfK linear} of $\mathbf{L}_{\Delta_T,K}$ and Triangular's inequality we derive \begin{align}\label{eq prf thm2 2}\big(\E\big[\big\|\widetilde{f}^K_{T,\Delta_T}-\mathbf{L}_{\Delta_T,K} \|_{L_p(\mathcal{D})}^p\big]\big)^{1/p}\leq \mathfrak{C}T^{-\alpha(s,p,\pi)},\end{align} where $\mathfrak{C}$ depends on $\vartheta,s,\pi,p,\mathfrak{M},\phi,\psi,$ and $K$. By \eqref{eq Taylor Approx renewal}, we have \begin{align}\label{eq prf thm2 3}\big\|\mathbf{L}_{\Delta_T,K} -\mathbf{H}_{\Delta,f}^{\circ K}\big[\mathbf{P}_{\Delta_T}[f]\big\|_{L_p(\mathcal{D})}\leq\mathfrak{C}\Delta_T^{K+1},\end{align} where $\mathfrak{C}$ depends on $\vartheta$, $\mathfrak{O}$ and $\mathfrak{M}$. For the last term we use the fixed point theorem's approximation, first we have to relate the $L_p$ norm with the Sobolev one. Triangular's inequality ensures that if $f$ is in $\mathcal{B}^s_{{\pi} \infty}(\mathcal{D})$ then $\mathbf{H}_{\Delta,f}^{\circ K}\big[\mathbf{P}_{\Delta_T}[f]\big] -f$ is in $\mathcal{B}^s_{{\pi} \infty}(\mathcal{D})$. It follows using Sobolev embeddings \eqref{eq Sobolev embed} that \begin{align*}\big\|\mathbf{H}_{\Delta,f}^{\circ K}\big[\mathbf{P}_{\Delta_T}[f]\big] -f\big\|_{L_p(\mathcal{D})}&\leq\big\|\mathbf{H}_{\Delta,f}^{\circ K}\big[\mathbf{P}_{\Delta_T}[f]\big] -f\big\|_{\mathcal{B}^s_{{\pi} \infty}(\mathcal{D})}.\end{align*} We now use the approximation given by the Banach fixed point theorem \begin{align*}\big\|\mathbf{H}_{\Delta,f}^{\circ K}\big[\mathbf{P}_{\Delta_T}[f]\big] &-f\big\|_{\mathcal{B}^s_{{\pi} \infty}(\mathcal{D})}\leq\mathfrak{K}(\Delta_T)^{K} \big\|\mathbf{H}_{\Delta,f}\big[\mathbf{P}_{\Delta_T}[f]\big] -\mathbf{P}_{\Delta_T}[f]\big\|_{\mathcal{B}^s_{{\pi} \infty}(\mathcal{D})}.\end{align*} After replacing $\mathbf{H}_{\Delta,f}\big[\mathbf{P}_{\Delta_T}[f]\big] $ by its expression and using triangular's inequality we have \begin{align*}\big\|\mathbf{H}_{\Delta,f}\big[\mathbf{P}_{\Delta_T}[f]\big] -\mathbf{P}_{\Delta_T}[f]\big\|_{\mathcal{B}^s_{{\pi} \infty}(\mathcal{D})}\leq \mathfrak{C}\Delta_T,\end{align*}which leads to \begin{align}\label{eq prf thm2 4}\big\|\mathbf{H}_{\Delta,f}^{ \circ K}\big[\mathbf{P}_{\Delta_T}[f]\big] -f\big\|_{L_p(\mathcal{D})}\leq \mathfrak{C}\Delta_T\mathfrak{K}(\Delta_T)^{K},\end{align} $\mathfrak{C}$ depends on $\vartheta,\mathfrak{M},\mathfrak{O},K$. We conclude by injecting \eqref{eq contract 0}, \eqref{eq prf thm2 2}, \eqref{eq prf thm2 3} and \eqref{eq prf thm2 4} in \eqref{eq prf thm2 1} and taking the supremum in $\vartheta$ over the compact set $\Theta$.

\

We now control $\E\big [\|\widehat{f}^K_{T,\Delta_T}-\widetilde{f}^K_{T,\Delta_T}\|_{L_p(\mathcal{D})}^p\big] $, the triangle inequality leads to
\begin{align*}
\big(\E\big[\|\widehat{f}^K_{T,\Delta_T}-&\widetilde{f}^K_{T,\Delta_T}\|_{L_p(\mathcal{D})}^p \big]\big)^{1/p}\\&\leq \sum_{m=1}^{K+1}\big(\E\big[\|\big(l_m(\Delta_T,\widehat{\vartheta_T})-l_m(\Delta_T,\vartheta)\big)\widehat{P_{\Delta_T,m}}\|_{L_p(\mathcal{D})}^p \big]\big)^{1/p},\end{align*} where $\widehat{P_{\Delta_T,m}}$ does not depend on $\vartheta$ (see \eqref{eq est Coeffconvol}). Cauchy-Schwarz inequality leads to \begin{align*}\E\big[\|\big(l_m(\Delta_T,\widehat{\vartheta_T})-&l_m(\Delta_T,\vartheta)\big)\widehat{P_{\Delta_T,m}}\|_{L_p(\mathcal{D})}^p \big]^2\\&\leq \E\Big[\big|l_m(\Delta_T,\widehat{\vartheta_T})-l_m(\Delta_T,\vartheta)\big|^{2p}\Big] \E\Big[\big\|\widehat{P_{\Delta_T,m}}\big\|_{L_{p}(\mathcal{D})}^{2p}\Big] ,\end{align*} where using part 1) of Theorem \ref{thm Renewal 2}, the triangle inequality and that $T\geq1$ we have
\begin{align}\E\Big[\big\|\widehat{P_{\Delta_T,m}}\big\|_{L_{p}(\mathcal{D})}^{2p}\Big]& \leq \E\Big[\|\widehat{P_{\Delta_T,m}}-\mathbf{P}_{\Delta_T}[f]^{\star m}\|_{L_{p}(\mathcal{D})}^{2p}\Big]+\|\mathbf{P}_{\Delta_T}[f]^{\star m}\|_{L_{p}(\mathcal{D})}^{2p}\nonumber\\&\leq \mathfrak{C}T^{-2{\alpha(s,p,\pi)} p}+\mathfrak{M}^{2p}\leq \mathfrak{C}\label{eq thm bound}\end{align} where $\mathfrak{C}$ depends on $s,\pi,p,\mathfrak{M},\phi,\psi,\vartheta$. We conclude the proof with the following Lemma, proof of which is given in the Appendix.
\begin{lemma}\label{lem lm ros} Work under Assumptions \ref{ass param tau} and \ref{ass queue tau}. We have for all $r\geq2$ \begin{align*}\E\big[|l_m(\Delta_T,\widehat{\vartheta_T})-l_m(\Delta_T,\vartheta)&|^{r}\big]\leq \mathfrak{C}\big(T^{1-r}+T^{-r/2}\big)\end{align*} where $\mathfrak{C}$ depends on $r,\mathfrak{A},\mathfrak{a},\vartheta.$\end{lemma} It follows from \eqref{eq thm bound} and Lemma \ref{lem lm ros} applied with $r=2p$ that
\begin{align*}
\E\big[\|\widehat{f}^K_{T,\Delta_T}-\widetilde{f}^K_{T,\Delta_T}&\|_{L_p(\mathcal{D})}\big]^{1/p}\leq \mathfrak{C}\big(T^{1-1/(2p)}+T^{-1/2}\big) ,\end{align*} where $\mathfrak{C}$ depends on $s,\pi,p,\mathfrak{M},\phi,\psi,\mathfrak{A},\mathfrak{a},\vartheta.$ We deduce for $p\geq 1$ \begin{align*}\underset{\vartheta\in \Theta }{\sup}\underset{f\in\mathcal{F}(s,{\pi},\mathfrak{M})}{\sup}\big(\E\big[\|\widehat{f}^K_{T,\Delta_T} \big(\widehat{\vartheta}\big)-\widehat{f}^K_{T,\Delta_T}&\|_{L_p(\mathcal{D})}^p\big]\big)^{1/p}\\&\leq \mathfrak{C}\big(T^{-(1-1/(2p))}+T^{-1/2}\big)\end{align*} where $\mathfrak{C}$ depends on $s,\pi,p,\mathfrak{M},\phi,\psi,\mathfrak{A},\mathfrak{a},K$. It is negligible compared to $T^{-{\alpha(s,p,\pi)}}$ since ${\alpha(s,p,\pi)}\leq1/2$. The proof of Theorem \ref{thm Renewal 2} is now complete.

\section*{Appendix}
\subsection*{Proof of Proposition \ref{PropDefOperator}}

Let $x\in\R$, we have by stationarity
\begin{align*}
\PP(\mathbf{D}^{\Delta} X_{S_1}\leq x)&=\PP(X_\Delta\leq x|X_\Delta\ne0)\\&=\sum_{m=0}^\infty\PP(X_\Delta\leq x|R_\Delta=m,R_\Delta\ne0) \PP(R_\Delta=m) \\
&=\sum_{m=1}^\infty p_m(\Delta) \PP(X_\Delta\leq x|R_\Delta=m) \end{align*} where
$
\PP(X_\Delta\leq x|R_\Delta=m)=\int_{-\infty}^xf^{\star m}(y)dy$ for $m\geq1$. It follows
\begin{align*}
\PP(\mathbf{D}^{\Delta} X_{S_1}\leq x)&=\int_{-\infty}^x \mathbf{P}_\Delta[f](y)dy.
\end{align*}

\subsection*{Proof of Lemma \ref{lem pmControl}}
We start with the second assertion. For $m\geq 1$ we have $$  p_m(\Delta)=\frac{\PP(R_\Delta=m)}{1-\PP(R_\Delta=0)}.$$ First we derive the lower bound \begin{align}\label{eq pmC1}1-\PP(R_\Delta=0)&=1-\PP(J_1\geq \Delta)
  \geq\frac{1}{\mu}\int_0^\Delta1-F(\Delta)dx\geq \frac{\Delta}{2\mu}
 \end{align} since $F$ is a cumulative distribution function; it is positive, increasing and continuous with $F(0)=0$. Then
 there exists $\Delta_1$ such that for all $\Delta\leq\Delta_1$ we have $F(\Delta)\leq\frac{1}{2}$. Second we have for all $m\geq1$
  \begin{align*}
\PP(R_\Delta=m)
&\leq \PP(J_1+\ldots+J_{m}\leq \Delta)=\int_0^\Delta \tau_0\star \tau^{\star m-1}(x)dx,
\end{align*}
where for all $x\in [0,\Delta]$ \begin{align*}
\tau_0&\star \tau^{\star m-1}(x)=x^{m-1}\int_0^1\tau_0(xt_1)\int_0^{1-t_1}\tau(xt_2)\ldots \\ &\int_0^{1-t_1-\ldots-t_{m-2}}\tau(xt_{m-1})\tau(x(1-t_1-\ldots-t_{m-2}-t_{m-1}))dt_1\ldots dt_{m-1}.\end{align*} We derive
\begin{align*}
\tau_0&\star \tau^{\star m-1}(x)\nonumber\\&\leq x^{m-1} \underset{t\in[0,x]}{\sup}\tau_0(t)\big(\underset{t\in[0,x]}{\sup}\tau(t)\big)^{m-1}\int_0^1\int_0^{1-t_1}\ldots \int_0^{1-t_1-\ldots-t_{m-2}}dt_1\ldots dt_{m-1}\nonumber\\
&\leq \frac{1}{\mu}\big(\underset{t\in[0,\Delta]}{\sup}\tau(t)\big)^{m-1} \frac{x^{m-1}}{(m-1)!},
\end{align*} since $$\int_0^1\int_0^{1-t_1}\ldots \int_0^{1-t_1-\ldots-t_{m-2}}dt_1\ldots dt_{m-1}=\frac{1}{(m-1)!}.$$
It follows that \begin{align}\label{eq pmC2}
\PP(R_\Delta=m)&\leq \frac{1}{\mu}\big(\underset{t\in[0,\Delta]}{\sup}\tau(t)\big)^{m-1}\frac{\Delta^{m}}{m!}.
\end{align} Since $\tau$ is continuous, there exists $\Delta_2$ such that $$\underset{t\in[0,\Delta_2]}{\sup}\tau(t)\leq 2\tau(0).$$ Taking $\Delta_0=\Delta_1 \wedge \Delta_2$, \eqref{eq pmC1} and \eqref{eq pmC2} lead to the second assertion. The first one is straightforward from the previous computations.

\subsection*{Proof of Proposition \ref{prop inverseTronqueeK}}

According to the definition of $\mathbf{L}_{\Delta,K}$ inequality \eqref{eq Taylor Approx renewal} is immediate. The dependency in $\tau(0)$ and $\mathfrak{M}$ of the constant is a consequence of Lemma \ref{lem pmControl}, part 2) of Proposition \ref{prop complet contract} and Lemma \ref{lem loicontY boule besov}.
 A rearrangement of the terms enables to write $\mathbf{L}_{\Delta,K}$ as a sum of increasing powers of ${\mathbf{P}_\Delta[f]}^{\star m}$. Thus we have to prove that only the $K+1$ first convolution powers of ${\mathbf{P}_\Delta[f]}$ intervene and that the coefficient $l_m(\Delta)$ in front of ${\mathbf{P}_\Delta[f]}^{\star m}$ in the rearrangement satisfies $$\big|l_m(\Delta)\big|\leq \mathfrak{C}_{\tau(0)}\Delta^{m-1}.$$

 For that we show that for all $L\geq 1$ the Taylor expansion of order $L$ in $\Delta$ of $\mathbf{H}_{\Delta,f}^{\circ K}\big[\mathbf{P}_\Delta[f]\big]$, that we denote $\widetilde{\mathbf{L}}_{\Delta,K,L}$, only depends on ${\mathbf{P}_\Delta[f]}^{\star m}$, $m=1,\ldots,L+1$ with coefficients such that $\widetilde l_{m,K}(\Delta)\leq \mathfrak{C}_{\tau(0)}\Delta^{m-1}.$ We prove the result by induction on $K$. For $K=1$ we immediately have the result by Lemma \ref{lem pmControl} since \begin{align*}
   \mathbf{H}_{\Delta,f}\big[\mathbf{P}_\Delta[f]\big]&= 2\mathbf{P}_\Delta[f]-\sum_{m=1}^\infty p_m(\Delta)\mathbf{P}_\Delta[f]^{\star m},
  \end{align*} it follows that $$\widetilde{\mathbf{L}}_{\Delta,L,1}=(2-p_1(\Delta)\mathbf{P}_\Delta[f]-\sum_{m=2}^{L+1}p_m(\Delta)\mathbf{P}_\Delta[f]^{\star m}$$ with $\widetilde l_{1,1}(\Delta)=(2-p_1(\Delta))\leq 2$ and $\widetilde l_{m,1}(\Delta)=p_m(\Delta)\leq \mathfrak{C}_{\tau(0)}\Delta^{m-1}$. Then using the definition of $\mathbf{H}_{\Delta,f}$ we have \begin{align*}\mathbf{H}_{\Delta,f}^{\circ (K+1)}\big[\mathbf{P}_\Delta[f]\big]=\mathbf{P}_\Delta[f]+\mathbf{H}_{\Delta,f}^{\circ K}\big[\mathbf{P}_\Delta[f]\big]-\sum_{m=1}^\infty p_m(\Delta)\Big(\mathbf{H}_{\Delta,f}^{\circ K}\big[\mathbf{P}_\Delta[f]\big]\Big)^{\star m}.\end{align*} The induction hypothesis and Lemma \ref{lem pmControl}, with part 2) of Proposition \ref{prop complet contract} which ensures that $\mathbf{H}_{\Delta,f}^{\circ K}\big[\mathbf{P}_\Delta[f]\big]\in\mathcal{H}(s,\pi,\mathfrak{O},\mathfrak{N}) $, lead to
  \begin{align*}
  \widetilde{\mathbf{L}}_{\Delta,L,K+1}&=\mathbf{P}_\Delta[f]+ \widetilde{\mathbf{L}}_{\Delta,L,K}-\sum_{m=1}^{L+1} p_m(\Delta)\Big(\widetilde{\mathbf{L}}_{\Delta,L,K}\Big)^{\star m}\\
  &=\mathbf{P}_\Delta[f]+ \sum_{m=1}^{L+1}\widetilde l_{m,L}(\Delta)\mathbf{P}_\Delta[f]^{\star m}-\sum_{m=1}^{L+1} p_m(\Delta)\Big(\sum_{m'=1}^{L+1}\widetilde l_{m',L}(\Delta)\mathbf{P}_\Delta[f]^{\star m'}\Big)^{\star m}\\&=
  \sum_{m=1}^{L+1}\widetilde l_{m,L+1}(\Delta)\mathbf{P}_\Delta[f]^{\star m},
  \end{align*}
  where $  \widetilde l_{1,L+1}(\Delta)=1$ and \begin{align*}
  \widetilde l_{m,L+1}(\Delta)&=\widetilde l_{m,L}(\Delta)-\sum_{k=1}^mp_k(\Delta)\sum_{n_1+\ldots+n_k=m}\widetilde l_{n_1,L}(\Delta)\ldots \widetilde l_{n_k,L}(\Delta)
  \end{align*} which we bound with Lemma \ref{lem pmControl} and the induction hypothesis by \begin{align*}
  \big|\widetilde l_{m,L+1}(\Delta)\big|&\leq \mathfrak{C}\Big(\Delta^{m-1}+\sum_{k=1}^m\Delta^{k-1}\sum_{n_1+\ldots+n_k=m}\widetilde \Delta^{n_1-1}\ldots \Delta^{n_k-1}\Big)\\
  &=\mathfrak{C}\Big(\Delta^{m-1}+m\Delta^{m-1}\Big)\leq\mathfrak{C}\Delta^{m-1},
  \end{align*} where $\mathfrak{C}$ is a positive constant depending on $\tau(0)$ and $K$. We conclude the proof having $L=K$ and $l_m(\Delta)=\widetilde l_{m,K}$ for $m=1,\ldots,K+1$.

\subsection*{Proof of Lemma \ref{lem lm ros}}

\subsubsection*{Preliminary}

\begin{lemma}\label{lem theta ros}Work under assumptions \ref{ass queue tau} and \ref{ass param tau}, for all $r\geq2$
\begin{align*}\E\big[|\widehat{\vartheta_T}-\vartheta|^r\big]&\leq \mathfrak{C}\big(T^{1-r}+T^{-r/2}\big),\end{align*} where $\mathfrak{C}$ depends on $r,\mathfrak{A},\mathfrak{a},\vartheta$ and $\widehat{\vartheta_T}$ is defined in Definition \ref{def est corr K}.\end{lemma}
\begin{proof}
Let $r>2$, the proof is a consequence of Proposition 5.5 of Dedecker \textit{et al.} \cite{Doukhan} which is a Rosenthal type inequality for dependent data. Define $$S_T=\sum_{i=1}^{\lfloor T \Delta^{-1}\rfloor}Y_1$$ where $S_0=X_0=0$ and the $Y_i=\mathds{1}_{\mathbf{D}^\Delta X_i\ne0}-q(\vartheta)$ are centered identically distributed random variables bounded by 1. To apply Proposition 5.5 of \cite{Doukhan} we have to verify that $(Y_i)$ is a sequence of $\theta_{1,\infty}-$dependent random variables. For that Proposition 2.3 of \cite{Doukhan} ensures that it is sufficient to have a $\theta-$dependent sequence which is defined as follows with notation of \cite{Doukhan}; Let $\Gamma(u,v,k)$ be the set of $(i,j)$ in $\Z^u\times\Z^v$ such that $$i_1<\ldots<i_u\leq i_u+k<j_1<\ldots<j_v,$$ we have to show that for all $f \in\mathcal{F}_u$ the set of bounded function from $\R^u$ to $\R$ and for all $g\in \mathcal{G}_v$ the set of Lipschitz function from $\R^v$ to $\R$ with Lipschitz coefficient denoted $\mbox{Lip}g$ the sequence $\theta(k)$ defined as $$\theta(k)=\underset{u,v}{\sup}\underset{(i,j)\in\Gamma(u,v,k)}{\sup}\underset{f\in \mathcal{F}_u,g\in\mathcal{G}_v}{\sup} \frac{\big|Cov\big(f(Y_{i_1},\ldots,Y_{i_u}),g(Y_{j_1},\ldots,Y_{j_v})\big)\big|}{v\|f\|_\infty\mbox{Lip}g}$$ tends to 0. We denote as $Y_i$ and $Y_j$ respectively $(Y_{i_1},\ldots,Y_{i_u})$ and $(Y_{j_1},\ldots,Y_{j_v})$, and due to the fact that $X$ is a renewal process $Y_i$ and $Y_j$ are independent if there exists $r$ such that $i_u< r <j_1$ and $Y_r=1-p(\Delta)$ \textit{i.e} there is a jump between $Y_{i_u}$ and $Y_{j_1}$. We denote by $A$ the event ''there exists $r$ such that $i_u< r <j_1$ and $Y_r=1-p(\Delta)$\". It follows that \begin{align*}\big| Cov\big(f(X_i),g(X_j)\big)\big|&= \big|\E\big[\big(f(X_i)-\E[f(X_i)]\big)\big(g(X_j)-g(0_j)\big)\mathds{1}_{\{A\}}\big]\big|\\
&\leq 2\|f\|_\infty \mbox{Lip}g\E\big[\|X_j\|\mathds{1}_{\{A\}}\big]\\&\leq 2v\|f\|_\infty \mbox{Lip}g\PP(R_{k\Delta}\ne 0),\end{align*} since $\|X_j\|\leq v$ as the $Y_i$ are bounded by 1, for every $L_p$ norm $p\geq0$, and $\E[\mathds{1}_{\{A\}}]$ is bounded by $\PP(R_{k\mathfrak{T}}\ne 0)$. We immediately derive that $\theta(k)\leq 2\PP(R_{k\Delta}\ne 0)$ and by Assumption \ref{ass queue tau} we derive \begin{align}\label{eq prf lem th01}\theta(k)\leq \mathfrak{C}\exp(-\mathfrak{a} (k\Delta)^\mathfrak{g'}),\end{align} where $\mathfrak{g}<\mathfrak{g'}$, it tends to 0. We verify the hypothesis of Proposition 5.5 of \cite{Doukhan} and get for all $r>2$ \begin{align*}
\E\big[|S_T|^r\big]&\leq \mathfrak{C}\Big(\lfloor T \Delta^{-1}\rfloor\sum_{i=1}^{\lfloor T \Delta^{-1}\rfloor} i^{r-2}\theta(i)+\big(\lfloor T \Delta^{-1}\rfloor\sum_{i=1}^{\lfloor T \Delta^{-1}\rfloor}\theta(i)\big)^{r/2}\Big)
\end{align*} where $\mathfrak{C}$ depends on $r$. Since we have the upper bound \eqref{eq prf lem th01}, we derive applying \eqref{eq prf bernDep} with $k=0$ and $k=r-2$ $$\sum_{i=1}^{\lfloor T \Delta^{-1}\rfloor}\theta(i)\leq \mathfrak{C}\Delta\ \ \ \mbox{ and }\ \ \ \sum_{i=1}^{\lfloor T \Delta^{-1}\rfloor} i^{r-2}\theta(i)\theta(i)\leq \mathfrak{C}\Delta^{r-1}$$ where $\mathfrak{C}$ depends on $\mathfrak{A},\mathfrak{a},\vartheta.$ It follows $$\frac{1}{\lfloor T \Delta^{-1}\rfloor^r}\E\big[|S_T|^r\big]\leq \mathfrak{C} \big(T^{1-r}+T^{-r/2}\big),$$ where $\mathfrak{C}$ depends on $r,\mathfrak{A},\mathfrak{a},\vartheta.$ The case $r=2$ is a consequence of $$\E\big[|S_T|^2\big]=\frac{1}{\lfloor T \Delta^{-1}\rfloor}\V(Y_1)+\frac{2}{\lfloor T \Delta^{-1}\rfloor^2}\sum_{1\leq i<j\leq T}Cov(Y_i,Y_j)$$ and the upper bounds $\V(Y_1)\leq\mathfrak{C}\Delta$ where $\mathfrak{C}$ depends on $\vartheta$ and $$|Cov(Y_i,Y_{i+k})|\leq \mathfrak{C}\exp(-\mathfrak{a} k\Delta).$$ We derive $$\frac{1}{\lfloor T \Delta^{-1}\rfloor^2}\E\big[|S_T|^2\big]=\mathfrak{C}T^{-1}$$ where $\mathfrak{C}$ depends on $\mathfrak{A},\mathfrak{a},\vartheta.$ We conclude the proof using Assumption \ref{ass param tau}, for all $r\geq2$
\begin{align*}\E\big[|\widehat{\vartheta_T}-\vartheta|^r\big]&=\E\big[|q^{-1}(q(\widehat{\vartheta_T}))-q^{-1}(q(\vartheta))|^r\big]\\
&\leq \|q^{-1}\|_\infty\ \E\big[|q(\widehat{\vartheta_T})-q(\vartheta)|^r\big]\leq \mathfrak{C}\big(T^{1-r}+T^{-r/2}\big),\end{align*}where $\mathfrak{C}$ depends on $r,\mathfrak{A},\mathfrak{a},\vartheta.$
\end{proof}

\subsubsection*{Completion of the proof of Lemma \ref{lem lm ros}}
The remaining of the proof is now based on the fact that under Assumption \ref{ass param tau} the functions $\vartheta\rightarrow p_m(.,\vartheta)$ are Lipschitz continuous. We show that their derivative with respect to $\vartheta$ is bounded, we have for $m\geq1$ that
\begin{align}
\partial_\vartheta[p_m(\Delta,\vartheta)]=&\frac{1}{\int_0^\Delta\tau_2(z,\vartheta)dz}\Big(\int_0^\Delta \partial_\vartheta[\tau_2(.,\vartheta)\star \tau_1^{\star m-1}(.,\vartheta)](z) dz\nonumber\\&-\int_0^\Delta \partial_\vartheta[\tau_2(.,\vartheta)\star \tau_1^{\star m}(.,\vartheta)](z)dz\Big)\nonumber\\&- \frac{\int_0^\Delta\partial_\vartheta[\tau_2(z,\vartheta)]dz}{\Big(\int_0^\Delta\tau_2(z,\vartheta)dz\Big)^2}\Big(\int_0^\Delta \tau_2(x,\vartheta)\star \tau_1^{\star m-1}(.,\vartheta)(z) dz\nonumber\\&-\int_0^\Delta \tau_2(x,\vartheta)\star \tau_1^{\star m}(.,\vartheta)(z)dz\Big)\label{eq prf lem 10 03}
\end{align} where $\tau_2(.,\vartheta)/\mu$ is the density of $J_1$. Immediate induction gives for $m\geq1$ \begin{align}\partial_\vartheta[\tau_2(.,\vartheta)\star \tau_1^{\star m}(.,\vartheta)](z)=&\partial_\vartheta[\tau_2(.,\vartheta)]\star \tau_1^{\star m}(.,\vartheta)(z)\nonumber\\&+m\tau_2(.,\vartheta)\star \partial_\vartheta[\tau_1(.,\vartheta)]\star \tau_1^{\star m-1}(.,\vartheta)(z)\label{eq prf lem 10 04}\end{align} and \begin{align}\label{eq prf lem 10 05} \int_0^\Delta g^{\star m}(x)dx\leq \mathfrak{C}\Delta^{m-1}\end{align} for some constant $\mathfrak{C}$ and any bounded function $g$ supported by $(0,\infty)$. Moreover we have \begin{align*}\partial_\vartheta[\tau_2(z,\vartheta)]=&-\int_0^z \partial_\vartheta \tau_1(x,\vartheta)dx,\end{align*} and it follows from Assumption \ref{ass param tau} that for $\Delta$ small enough we have $\forall z\leq \Delta$
\begin{align}\label{eq prf lem 10 06}
0<\frac{\tau_1(0,\vartheta)}{2}&\leq \tau_1(z,\vartheta)\leq 2\tau_1(0,\vartheta),
\end{align} and that its derivative is bounded over $[0,\Delta]$. Finally we bound \eqref{eq prf lem 10 03}, using \eqref{eq prf lem 10 04} \eqref{eq prf lem 10 05} and \eqref{eq prf lem 10 06}, we get \begin{align*}
\big|\partial_\vartheta[p_m(\Delta,\vartheta)]\big|&\leq \mathfrak{C} \Delta^{m-1},
\end{align*} where $\mathfrak{C}$ continuously depends on $\vartheta$. Then taking the supremum in $\vartheta$ over the compact set $\Theta$ we derive
\begin{align*}
\big|\partial_\vartheta[p_m(\Delta,\vartheta)]\big|&\leq \mathfrak{C} \Delta^{m-1},
\end{align*} where $\mathfrak{C}$ is a positive constant independent of $\vartheta$. It follows that for $m\geq 1$, the functions $\vartheta\rightarrow p_m(.,\vartheta)$ are Lipschitz continuous and with Lemma \ref{lem theta ros} we derive \begin{align*}\E\big[\big|p_m(\Delta,\widehat{\vartheta_T}) -p_m(\Delta,\vartheta)\big|^r\big]&\leq \mathfrak{C}\Delta^{m-1}\E\big[\big|\widehat{\vartheta_T}-\vartheta\big|^r\big]\\&\leq \mathfrak{C} \big(T^{1-r}+T^{-r/2}\big),\end{align*}
where $\mathfrak{C}$ is a positive constant depending on $r,\mathfrak{A},\mathfrak{a},\vartheta.$ We conclude the proof using that $l_m(\Delta,\vartheta)=l(p_1(\Delta,\vartheta),\ldots,p_m(\Delta,\vartheta))$ where $l$ is Lipschitz in every argument and the argument are bounded by 1.

\section*{Acknowledgements}
This work is a part of the author's Ph.D thesis under the supervision of Marc Hoffmann whom I would like to thanks for his valuable remarks on this paper. The author's research is supported by a PhD GIS Grant.

\end{document}